\documentclass[11pt]{amsart}
\pdfoutput=1

\usepackage[left=2.8cm,right=2.8cm,bottom=3.1cm,top=3cm]{geometry}

\usepackage{shuffle}

\usepackage[utf8]{inputenc}
\usepackage{subfiles}
\usepackage{hyperref}
\usepackage{amsthm}
\usepackage{amsmath}
\usepackage{amssymb}
\usepackage{amsfonts}

\usepackage{bm}
\usepackage{enumitem}
\usepackage{mathtools} 
\usepackage{tikz}
\usepackage{tikz-cd}
\usepackage{subcaption}
\usepackage{makecell} 
\usetikzlibrary{hobby} 
\usetikzlibrary{decorations.markings} 
\tikzset{->-/.style={decoration={markings,mark=at position #1 with {\arrow{>}}},postaction={decorate}}} 

\usepackage{mathrsfs}

\newtheorem{theorem}{Theorem}[section]
\theoremstyle{definition}
\newtheorem{proposition}[theorem]{Proposition}
\newtheorem{lemma}[theorem]{Lemma}
\newtheorem{definition}[theorem]{Definition}
\newtheorem{remark}[theorem]{Remark}
\newtheorem{corollary}[theorem]{Corollary}

\newtheorem{example}[theorem]{Example}

\newtheorem{observation}[theorem]{Observation}

\theoremstyle{remark}

\newcommand{\Z}{\mathbb Z}
\newcommand{\Q}{\mathbb Q}
 
\newcommand{\C}{\mathbb C}

\renewcommand{\L}{\mathcal L}

\newcommand{\B}{\mathcal B}

\newcommand{\A}{\mathcal A}

\newcommand{\dX}{\mathrm{d}}

\DeclareMathOperator{\id}{id}

\DeclareMathOperator{\Li}{Li}

\DeclareMathOperator{\gr}{gr}

\DeclareMathOperator{\ev}{ev}

\DeclareMathOperator{\INV}{INV}

\DeclareMathOperator{\symb}{symb}
\DeclareMathOperator{\Mult}{Mult}

\usepackage[all]{xy}
\CompileMatrices
\def\cxymatrix#1{\xy*[c]\xybox{\xymatrix#1}\endxy}

\makeatletter
\newcommand*{\Relbarfill@}{\arrowfill@\Relbar\Relbar\Relbar}
\newcommand*{\xequal}[2][]{\ext@arrow 0055\Relbarfill@{#1}{#2}}
\makeatother

\title{Hopf algebras of multiple polylogarithms, and holomorphic 1-forms}

\author{Zachary Greenberg}
\address{Universit\"at Heidelberg\\
         Mathematiches Institut\\
         69120 Heidelberg Germany}
\email{zgreenberg@mathi.uni-heidelberg.de}

\author{Dani Kaufman}
\address{University of Copenhagen \\
         Department of Mathematical Sciences \\
         2100 Copenhagen \o, Denmark 
         \newline
         {\tt \url{https://sites.google.com/danikaufman/home}}}
\email{dk@math.ku.dk}

\author{Haoran Li}
\address{University of Maryland \\
         Department of Mathematics \\
         College Park, MD 20742-4015, USA}
\email{haoranli@umd.edu }

\author{Christian K. Zickert}
\address{University of Maryland \\
         Department of Mathematics \\
         College Park, MD 20742-4015, USA} 
\email{zickert@umd.edu}

\thanks{C.~Z.~was supported in part by DMS-1711405.\\
\newline
2020 {\em Mathematics Classification.} Primary 11G55. 
Secondary 19E15, 14D07, 32G20. 
}

\begin{document}
\begin{abstract}
We associate to a multiple polylogarithm a holomorphic 1-form on the universal abelian cover of its domain. We relate the 1-forms to the symbol and variation matrix and show that the 1-forms naturally define a lift of the variation of mixed Hodge structure associated to a polylogarithm. The results are conveniently described in terms of a variant $\mathbb H^{\symb}$ of Goncharov's Hopf algebra of multiple polylogarithms. In particular, we show that the association of a 1-form to a multiple polylogarithm induces a map from the Chevalley-Eilenberg complex of the Lie coalgebra of indecomposables of $\mathbb H^{\symb}$ to the de Rham complex.
\end{abstract}
\maketitle
\setcounter{tocdepth}{2}
\tableofcontents

\section{Summary of Results}\label{sec:Summary}
The \emph{multiple polylogarithm}~\cite{GoncharovPolylogsArithmetic} of \emph{weight} $n_1+\dots+n_d$ and \emph{depth} $d$ is defined by the power series
\begin{equation}\label{eq:PowerSeriesLi}
    \Li_{n_1,\dots,n_d}(x_1,\dots,x_d)=\sum_{k_1< \dots < k_d} \frac{x_1^{k_1}\cdots x_d^{k_d}}{k_1^{n_1}\cdots k_d^{n_d}}.
\end{equation}
It defines a multivalued (single valued on the universal cover) holomorphic function on the space 
\begin{equation}
    S_d(\C)=\left\{(x_1,\dots,x_d)\in\C^d\bigm \vert x_i\neq 0,\,\prod_{r=i}^jx_r\neq 1 \text{ for all }1\leq i\leq j\leq d\right\}.
\end{equation}
There is a simple model for the universal abelian cover of $S_d(\C)$, namely
\begin{equation}
\widehat S_d(\C)=\left\{(u_i,v_{i,j})\in \C^{d+\binom{d+1}{2}}\bigm\vert \exp(\sum_{r=i}^j u_i)+\exp(v_{i,j})=1 \text{ for all } 1\leq i\leq j\leq d\right\}.
\end{equation}
Our goal is to associate to a multiple polylogarithm $\Li_{\mathbf n}(\mathbf x)$ of depth $d$ a holomorphic 1-form $w_{\mathbf n}$ on $\widehat S_d(\C)$, and to study the structure and properties of these forms. Most of our results are defined using a Hopf algebra $\mathbb H^{\symb}$ of symbolic polylogarithms.

\begin{remark}
When $d=1$, $\widehat S_d(\C)$ is the space $\widehat\C=\{(u,v)\in\C^2\mid e^u+e^v=1\}$ introduced by Neumann~\cite{Neumann} and playing a prominent role e.g.~ in~\cite{ZickertAlgK,GaroufalidisThurstonZickert,Zickert_HolomorphicPolylogarithmsAndBlochComplexes}.
\end{remark}

\subsection{Hopf algebras of polylogarithms}
Motivated by attempts to construct the Hopf algebra of regular functions on the motivic Galois group, Goncharov~\cite{Goncharov_GaloisSymmetriesOfFundamentalGroupoidsAndNoncommutativeGeometry,GoncharovPeriodsMixedMotives} constructed several Hopf algebras related to polylogarithms and iterated integrals. The slightly modified variant $\mathbb H^{\symb}$ considered here is a free Hopf algebra with generators $[x_{i_1\to i_2},\dots,x_{i_d\to i_{d+1}}]_{n_1,\dots,n_d}$ and $[x_i]_0$ in one-to-one correspondence with functions
\begin{equation}
    \Li_{n_1,\dots,n_d}(x_{i_1\to i_2},\dots,x_{i_d\to i_{d+1}}),\qquad \log(x_i), \quad \text{ where }\qquad x_{i\to j}=\prod_{r=i}^{j-1} x_r.
\end{equation}
We stress that the variables $x_i$ are in order and form an unbroken sequence, e.g.~there are no generators corresponding to $\Li_{2,1}(x_2,x_1)$ or  $\Li_{2,2}(x_1,x_3)$. The coproduct is the one defined by Goncharov~\cite[Prop.~6.1]{Goncharov_GaloisSymmetriesOfFundamentalGroupoidsAndNoncommutativeGeometry} except that the \emph{inverse terms} (e.g.~$\Li_{2,1}(x_2^{-1},x_1^{-1})$) appearing in Goncharov's formula are inverted using a function $\INV$, which is a variant of Goncharov's \emph{inversion formula}~\cite[Sec.~2.6]{Goncharov_MultiplePolylogarithmsAndMixedTateMotives}. We refer to Section~\ref{sec:HopfAlgebras} for the precise definition. 

A key property of $\mathbb H^{\symb}$ is that the coproduct of $[x_1,\dots,x_n]_{n_1,\dots,n_d}$ can be written entirely in terms of \emph{contractions}. For $\mathbf i= (i_1,\dots,i_{d+1})$ a strictly increasing subsequence with $i_{d+1}\leq n+1$, the contraction of $(x_1,\dots,x_n)$ by $\mathbf{i}$ is given by
\begin{equation}
    \mathbf i(x_1,\dots,x_n)=(x_{i_1\to i_2},\dots, x_{i_d\to i_{d+1}}).
\end{equation}
As a result we obtain that for any \emph{contraction system}, i.e.~a collection of sets $\{X_k\}_{k=1}^\infty$ with contractions $\mathbf i\colon X_n\to X_d$  (see Definition~\ref{def:ContractionSystem}) we have a Hopf algebra freely generated by symbols $[\alpha]_{n_1,\dots,n_d}$ for $\alpha\in X_d$ and $[\alpha]_0$ for $\alpha\in X_1$. The coproduct is induced by that of $\mathbb H^{\symb}$. We have natural contraction systems associated to the following situations (see Section~\ref{sec:ContractionHopfAlgebras}):
\begin{itemize}
    \item A field $F$.
    \item A field $F$ together with a torsion free $\Z$-extension $\pi\colon E\to F^*$.
    \item An open subset of a complex manifold $M$.
    \item Any simplicial set $S_\bullet$.
\end{itemize}
We thus obtain Hopf algebras (a sheaf of Hopf algebras in the third case)
\begin{equation}\label{eq:AssHopfAlgs}
    \mathbb H^{\symb}(F),\qquad \widehat{\mathbb H}^{\symb}_E(F),\qquad \widehat{\mathbb H}_M,\qquad \mathbb H^{\symb}_{S_\bullet}.
\end{equation}
The ``hat'' indicates that the contraction system involves a universal abelian cover (or an algebraic analogue). 

For any graded Hopf algebra $H$ there is an associated Lie coalgebra $L=\frac{H_{>0}}{H_{>0}H_{>0}}$ of indecomposables. We thus have a Lie coalgebra $\mathbb L^{\symb}$ along with Lie coalgebras
\begin{equation}\label{eq:LieCoalgebras}
     \mathbb L^{\symb}(F),\qquad \widehat{\mathbb L}^{\symb}_E(F),\qquad \widehat{\mathbb L}_M,\qquad \mathbb L^{\symb}_{S_\bullet}.
\end{equation}
\begin{remark} We stress that $\mathbb H^{\symb}$ is a Hopf algebra over $\Q$, but $\mathbb L^{\symb}$ is a Lie coalgebra over $\Z$, as are all the Lie coalgebras in~\eqref{eq:LieCoalgebras}.
\end{remark}

\subsection{The forms}

Let $\Omega^*$ denote the algebraic de Rham complex for the polynomial ring generated over $\Q$ by free variables $u_i$ and $v_{i,j}$ (where $1\leq i\leq j\in\Z$). Note that $\omega\in\Omega^k$ can be canonically realized as a holomorphic $k$-form on $\widehat S_d(\C)$ for any large enough $d$. 

In Section~\ref{sec:Forms} we construct a map
\begin{equation}
    w\colon\mathbb H^{\symb}\to \Omega^1
\end{equation}
which is such that the image of $\mathbb H^{\symb}(d)$, the subalgebra generated by terms involving only $x_i$ for $i\leq d$, has a canonical realization in $\Omega^1(\widehat S_d(\C))$. 
\begin{remark}The map $w$ factors through the \emph{symbol map} and can be easily expressed in terms of the \emph{symbol modulo products} (see Lemma~\ref{lemma:wFac}). 
\end{remark}
\begin{example}\label{ex:FormsForClassicalPolylogs}
In depth 1 we have ($v_{i}$ is shorthand for $v_{i,i}$)
\begin{equation} w[x_i]_0=\dX u_i,\quad w[\prod_{r=i}^jx_r]_1=-\dX v_{i,j},\quad w[x_1]_{n\geq 2} = (-1)^{n}\frac{1}{n!} u_1^{(n-2)}(u_1~\dX v_1 - v_1~\dX u_1).
\end{equation}
We note that $(n-1)w[x_1]_n$ are exactly the forms that appeared in  ~\cite{Zickert_HolomorphicPolylogarithmsAndBlochComplexes}.
\end{example}
\begin{example} We have
\begin{equation}
w[x_1,x_2]_{1,1} = \frac{1}{2}\big(v_{1,2}~\dX u_1 + (v_2-v_{1,2}) ~\dX v_1 + (v_{1,2}-v_1)~\dX v_2 - (u_1-v_1+v_2)~\dX v_{1,2}\big).
\end{equation}
\end{example}
In higher weight and depth the forms are most easily computed via a recurrence relation. Such a relation first appeared in Greenberg's thesis~\cite{ZackThesis}. We express it more cleanly in terms of the variation matrix (see Section~\ref{sec:Recurrence}).

Our main structural result about the 1-forms is the following.
\begin{theorem}\label{thm:ComDiagram} The map $w$ kills products and induces a commutative diagram
\begin{equation}
    \xymatrixcolsep{5pc}\cxymatrix{{\mathbb L^{\symb}\ar[r]^{\delta}\ar[d]^-w&\wedge^2(\mathbb L^{\symb})\ar^{\delta\wedge\id-\id\wedge\delta}[r]\ar[d]^{w\wedge w}&\wedge^3(\mathbb L^{\symb})\ar[r]\ar[d]^{w\wedge w\wedge w}&\dots\\\Omega^1\ar[r]^-\dX&\Omega^2\ar[r]^-\dX&\Omega^3\ar[r]&\dots}}
\end{equation}
In particular, we obtain a chain map from $\wedge^*(\mathbb L^{\symb}(d))$ to the de Rham complex of $\widehat S_d(\C)$.
\end{theorem}

The following is an immediate corollary of Theorem~\ref{thm:ComDiagram}.
\begin{theorem} The map $w$ induces a morphism of complexes of sheaves $\wedge^*(\widehat{\mathbb L}^{\symb}_M)\to\Omega^*_M$, where $\Omega^*_M$ is the holomorphic de Rham complex on a complex manifold $M$.
\end{theorem}

\begin{remark} We expect that there is a quotient, $\widehat{\mathbb L}_M$, of $\widehat{\mathbb L}_M^{\symb}$ such that the hypercohomology of $M$ with coefficients in the complex $\wedge^*(\widehat{\mathbb L}_M)_n$ computes the \emph{integral} motivic cohomology groups $H^*_{\mathcal M}(M;\Z(n))$ and that the above map induces the de Rham realization.
\end{remark}

\begin{remark}
  A quotient $\mathbb L(F)$ of $\mathbb L^{\symb}(F)$ conjecturally computing the \emph{rational} motivic cohomology groups of $F$ was constructed in~\cite{GKLZLieCoalg}. We expect that there is a quotient $\widehat{\mathbb L}(F)$ of $\widehat{\mathbb L}_E^{\symb}(F)$ computing \emph{integral} cohomology, whose construction should be similar to the construction of lifted Bloch complexes in~\cite{Zickert_HolomorphicPolylogarithmsAndBlochComplexes}.
\end{remark}

\subsection{Variation matrices and mixed Hodge structure}
The monodromy of a multiple polylogarithm can be computed using a variation matrix. For the classical polylogarithms the variation matrix was defined by Deligne and Beilinson~\cite{Deligne_InterpretationMotiviqueDeLaConjectureDeZagierReliantPolylogarithmesEtRegulateurs} and for the multiple polylogarithms by Zhao~\cite{Zhao_MultipleZetaFunctionsMultiplePolylogarithmsAndTheirSpecialValues}. For example, for $\Li_3(x)$ and $\Li_{1,1}(x,y)$ the variation matrices are given by
\begin{equation}\label{eq:VarMatrices}
\begin{pmatrix}
1&&&\\\Li_1(x)&1&&\\\Li_2(x)&\log(x)&1&\\\Li_3(x)&\frac{1}{2}\log(x)^2&\log(x)&1
\end{pmatrix},
\qquad
\begin{pmatrix}
1&&&\\\Li_1(y)&1&&\\\Li_1(xy)&0&1&\\\Li_{1,1}(x,y)&\Li_1(x)&\Li_1(y)-\Li_1(x^{-1})&1
\end{pmatrix}.
\end{equation}
We give a purely symbolic definition of variation matrices as matrices with entries in $\mathbb H^{\symb}$ and show that
\begin{equation}\label{eq:DeltaVTIntro}
    \Delta(V^T)=V^T\otimes V^T.
\end{equation}
This result plays an important role in the proof of Theorem~\ref{thm:ComDiagram}.
It is known (see e.g.~\cite{Zhao_MultipleZetaFunctionsMultiplePolylogarithmsAndTheirSpecialValues}) that the variation matrix $V$ for a depth $d$ polylogarithm satisfies a differential equation $\dX V=\omega V$, where $\omega$ is a matrix of $1$-forms on $S_d(\C)$. Also, if  $V$ is an $N\times N$ matrix, then $V$ defines a variation of mixed Hodge structure on $S_d(\C)$ with Hodge filtration induced by the standard filtration on $\C^N$, weight filtration coming from the span of column vectors of $V$, and connection form $\nabla=\dX-\omega$.

The form $\omega$ is exact on $\widehat S_d(\C)$ with a canonical primitive $\Omega$. In Section~\ref{sec:variationOfHodgeStructure} we show that $\widehat V=e^{-\Omega} V$ defines a variation of mixed Hodge structure on $\widehat S_d(\C)$ whose connection form is given by $\dX-\widehat\omega$ where $\widehat\omega$ is determined by the 1-forms (rescaled by $(n-1)$). This gives a Hodge theoretic interpretation of the 1-forms appearing in Zickert~\cite{Zickert_HolomorphicPolylogarithmsAndBlochComplexes} (which as mentioned in Example~\ref{ex:FormsForClassicalPolylogs} are scaled by $(n-1)$). We note that the matrix $\widehat V$ is much sparser than $V$.

\subsection{Structure of the paper}
In Section~\ref{sec:HopfAlgebras} we define $\mathbb H^{\symb}$, the associated Hopf algebras~\eqref{eq:AssHopfAlgs}, and the symbolic variation matrix $V$. We also define a symbolic derivation corresponding to the standard derivative. The main results stated in this section are that $\mathbb H^{\symb}$ is a Hopf algebra, that $\Delta V^T=V^T\otimes V^T$, and that $V$ satisfies a symbolic differential equation $\dX V=\omega V$ where $\omega =\dX V_1$ and $V_1$ is the weight $1$ part of $V$. We also show that the antipode satisfies $S(V)=V^{-1}$. The proof of coassociativity uses properties of Goncharov's Hopf algebra of iterated integrals and are deferred to Sections~\ref{sec:IteratedIntegrals} and \ref{sec:DeltaCommutesWithINV}. In Section~\ref{sec:Tensors} we recall the Hopf algebra structure on the tensor algebra and the symbol map, which is used to define the 1-forms. In Section~\ref{sec:Forms} we define the map $w\colon\mathbb H^{\symb}\to\Omega^1$ and prove Theorem~\ref{thm:ComDiagram}. In Section~\ref{sec:variationOfHodgeStructure} we recall the variation of mixed Hodge-Tate structure associated to a multiple polylogarithm and show that it has a natural lift to the universal abelian cover. In Section~\ref{sec:IteratedIntegrals} we recall Goncharov's Hopf algebra of iterated integrals, and give an elementary proof that the coproduct is coassociative. We show that Goncharov's expression of iterated integrals in terms of multiple polylogarithms can be viewed as a morphism of Hopf algebras, and use this to prove coassociativity of the coproduct on $\mathbb H^{\symb}$. The most technical part, that $\INV$ commutes with $\Delta$, is relegated to Section~\ref{sec:DeltaCommutesWithINV}.

\section{Hopf algebras of symbolic polylogarithms}\label{sec:HopfAlgebras}
Let $x_1,x_2,\dots$ be free variables and let $x_{i\to j}=\prod_{r=i}^{j-1} x_r$.
Define $\mathbb H^{\symb}$ to be the free graded $\Q$-algebra generated by symbols $[x_{i_1\rightarrow i_2},\dots,x_{i_d\rightarrow i_{d+1}}]_{n_1,\dots,n_d}$ in weight $n_1+\dots+n_d$ and symbols $[x_i]_0$ in weight 1 (not 0). Here $\mathbf i=(i_1,\dots,i_{d+1})$ and $\mathbf n=(n_1,\dots,n_d)$ consist of positive integers, and $\mathbf i$ is strictly increasing. We define $[x_{i\to j}]_0=-[x_{i\to j}^{-1}]_0=\sum_{r=i}^{j-1}[x_r]_0$ and $[1]_0=0$.

As we shall see, $\mathbb H^{\symb}$ is a graded Hopf algebra. To define the coproduct, we introduce an auxillary  Hopf algebra $\overline{\mathbb H}^{\symb}$, which is generated by symbols as above (called \emph{regular symbols}) together with additional \emph{inverted symbols} $[x_{i_d\to i_{d+1}}^{-1},\dots,x_{i_1\to i_2}^{-1}]_{n_d,\dots,n_1}$. Note that the order of terms is reversed for inverted symbols.  
Formulas for the coproducts are given below with proofs of coassociativity deferred to Section~\ref{sec:IteratedIntegrals}.
We use Goncharov's generating series (see \cite{Goncharov_MultiplePolylogarithmsAndMixedTateMotives})
\begin{equation}
\begin{aligned}
[\mathbf y|\mathbf t]=[y_1,\dots,y_d|t_1,\dots,t_d]&=\sum_{n_i\geq 1}[y_1,\dots,y_d]_{n_1,\dots,n_d}t_1^{n_1-1}\dots t_d^{n_d-1}\\
\exp([y]_0t)&=\sum_{i=0}^\infty \frac{[y]_0^n}{n!}.
\end{aligned}
\end{equation}

\begin{definition}[{c.f.~\cite[Prop~6.1]{Goncharov_GaloisSymmetriesOfFundamentalGroupoidsAndNoncommutativeGeometry}}]\label{def:GoncharovCoproduct} Define a coproduct $\Delta\colon\overline{\mathbb H}^{\symb}\to\overline{\mathbb H}^{\symb}\otimes\overline{\mathbb H}^{\symb}$ by $\Delta([x_i]_0)=[x_i]_0\otimes 1+1\otimes [x_i]_0$ and
\begin{equation}\label{eq:GoncharovCoproductFormula}
\begin{aligned}
   \Delta([\mathbf y|\mathbf t])=\sum&[y_{i_1\to i_2},\dots,y_{i_k\to i_{k+1}}|t_{j_1},\dots ,t_{j_k}]\bigotimes\\
    &\prod_{\alpha=0}^k(-1)^{j_\alpha-i_\alpha}\exp([y_{i_\alpha\to i_{\alpha+1}}]_0t_{j_\alpha})[y_{j_\alpha-1}^{-1},y_{j_\alpha-2}^{-1},\dots,y_{i_\alpha}^{-1}|t_{j_\alpha}-t_{j_\alpha-1},\dots, t_{j_\alpha}-t_{i_\alpha}]\\
    &\hphantom{\prod_{\alpha=0} (-1)}[y_{j_\alpha+1},y_{j_\alpha+2},\dots,y_{i_{\alpha+1}-1}|t_{j_\alpha+1}-t_{j_\alpha},\dots,t_{i_{\alpha+1}-1}-t_{j_\alpha}].\\
\end{aligned}
\end{equation}
The sum is over all instances of $0=i_0\leq j_0<i_1\leq j_1<\dots <i_k\leq j_k<i_{k+1}=d+1$, and by definition we have $y_{i\to j}=\prod_{r=i}^{j-1}y_r$, $[\emptyset|\emptyset]=1$, $t_0=0$, and $y_0=1$.
\end{definition}

\begin{example}
In depth 1~\eqref{eq:GoncharovCoproductFormula} becomes
\begin{equation}
\Delta[y|t]=[y|t]\otimes\exp([y]_0t)+1\otimes[y|t].
\end{equation}
\end{example}
\begin{example}\label{ex:Coprod31} In depth 2, $\Delta[y_1,y_2|t_1,t_2]$ equals
\begin{multline}\label{eq:CoproductDepth2}
[y_1,y_2|t_1,t_2]\otimes\exp([y_1]_0t_1+[y_2]_0t_2)+[y_1y_2|t_1]\otimes\exp([y_1y_2]_0t_1)[y_2|t_2-t_1]\\
-[y_1y_2|t_2]\otimes\exp([y_1y_2]_0t_2)[y_1^{-1}|t_2-t_1]+[y_2|t_2]\otimes[y_1|t_1]\exp([y_2]_0t_2)+1\otimes[y_1,y_2|t_1,t_2].
\end{multline}
The coproduct of  $[y_1,y_2]_{n_1,n_2}$, is obtained as the coefficient of $t_1^{n_1-1}t_2^{n_2-1}$. For example,
\begin{equation}
\begin{aligned}
\Delta[y_1,y_2]_{3,1} = \big([y_1,y_2]_{3,1}\otimes 1  + [y_1,y_2]_{2,1}\otimes [y_1]_0 + [y_1,y_2]_{1,1}\otimes \frac{1}{2}[y_1]_0^2\big)\\
+\Big([y_1y_2]_3\otimes[y_2]_1+[y_1y_2]_2\otimes(-[y_2]_2+[y_2]_1[y_1y_2]_0)+\\
[y_1y_2]_1\otimes ([y_2]_3-[y_2]_2[y_1y_2]_0+\frac{1}{2}[y_2]_1[y_1y_2
]_0^2)\Big)\\
-[y_1y_2]_1\otimes [y_1^{-1}]_3+[y_2]_1\otimes [y_1]_3+ 1\otimes[y_1,y_2]_{3,1} 
\end{aligned}
\end{equation}
\end{example}

\begin{theorem}[Proof in Section~\ref{sec:IAndPhi}]\label{thm:HbarHopf} The coproduct $\Delta$ is coassociative and thus endows $\overline{\mathbb H}^{\symb}$ with a graded Hopf algebra structure.
\end{theorem}

\begin{remark}
$\overline{\mathbb H}^{\symb}$ is connected with unit and counit given by inclusion of (resp.~projection onto) $\overline{\mathbb H}^{\symb}_0=\Q$. The antipode is given in Section~\ref{sec:Antipode}.
\end{remark}

\subsection{Goncharov's inversion formula and the coproduct on \texorpdfstring{$\mathbb H^{\symb}$}{Hsymb}}
In  \cite[Section~2.6]{Goncharov_MultiplePolylogarithmsAndMixedTateMotives}, Goncharov proved a relation between $\Li(y_1,\dots,y_d|t_1,\dots,t_d)$ and $\Li(y_d^{-1},\dots,y_1^{-1}|-t_d,\dots,-t_1)$. Motivated by this we define a map $\INV\colon\overline{\mathbb H}^{\symb}\to\mathbb H^{\symb}$, which fixes regular symbols and is defined inductively on inverted symbols by the power series formula
\begin{equation}\label{eq:GoncharovInversionFormula}
\begin{aligned}
&\INV([y_d^{-1},\cdots,y_1^{-1}|-t_d,\cdots,-t_1])\\
&=\sum_{j=0}^{d-1}(-1)^{d-1+j}\INV([y_j^{-1},\cdots,y_1^{-1}|-t_j,\cdots,-t_1])[y_{j+1},\cdots,y_d|t_{j+1},\cdots,t_d]\\
&+\sum_{j=1}^d\frac{(-1)^{d-1+j}}{t_j}\INV([y_{j-1}^{-1},\cdots,y_1^{-1}|-t_{j-1},\cdots,-t_1])[y_{j+1},\cdots,y_d|t_{j+1},\cdots,t_d]\\
&+\sum_{j=1}^d\Big(\frac{(-1)^{d+j}}{t_j}\INV([y_{j-1}^{-1},\cdots,y_1^{-1}|t_j-t_{j-1},\cdots,t_j-t_1])\\&\phantom{abcsdfsdfsdfssd}\exp([y_{1\to d+1}]_0t_j)[y_{j+1},\cdots,y_d|t_{j+1}-t_{j},\cdots,t_d-t_j]\Big)
\end{aligned}
\end{equation}
The induction starts with $\INV([y^{-1}|-t])=[y|t]+\frac{\exp([y]_0t)-1}{t}$.
\begin{example}\label{ex:INVDepth1}
In depth 1, to compute $\INV[y^{-1}]_n$ we extract the $t^{(n-1)}$ term of the power series.
\begin{equation}
\INV[y^{-1}]_n = (-1)^{n-1}[y]_{n} + \frac{(-1)^{n-1}}{n!}[y]_0^n
\end{equation}
\end{example}
\begin{example}
In depth 2, $\INV([y_2^{-1},y_1^{-1}|-t_2,-t_1])$ equals
\begin{multline}
-[y_1,y_2|t_1,t_2]+\INV([y_1^{-1}|-t_1])[y_2|t_2]+\frac{1}{t_1}[y_2|t_2]-\frac{1}{t_2}\INV([y_1^{-1}|-t_1])\\-\frac{1}{t_1}\exp([y_1y_2]_0t_1)[y_2|t_2-t_1]+\frac{1}{t_2}\INV([y_1^{-1}|t_2-t_1])\exp([y_1y_2]_0t_2).
 \end{multline}
Using Example~\ref{ex:INVDepth1} we obtain (by extracting the coefficient of $t_1^2$) that $\INV([y_2^{-1},y_1^{-1}]_{1,3})$ equals 
\begin{multline}
 -[y_1,y_2]_{3,1}+\left([y_1]_3+\frac{1}{6}[y_1]_0^3\right)[y_2]_1-\left(\frac{1}{6}[y_2]_1[y_1y_2]_0^3-\frac{1}{2}[y_2]_2[y_1y_2]_0^2+[y_2]_3[y_1y_2]_0-[y_2]_4\right)\\
+\left(\left([y_1]_3+\frac{1}{6}[y_1]_0^3\right)[y_1y_2]_0-3 \left([y_1]_4+\frac{1}{24}[y_1]_0^4\right)\right).
\end{multline}
\end{example}
The result below states that we obtain a coproduct $\Delta$ on $\mathbb H^{\symb}$ by applying $\INV$ to the inverted terms of $\Delta$.
\begin{theorem}[Proof in Section~\ref{sec:DeltaCommutesWithINV}]\label{thm:HHopf} The restriction of $\INV\circ\Delta$ to $\mathbb H^{\symb}$ is coassociative and endows $\mathbb H^{\symb}$ with a graded Hopf algebra structure.
\end{theorem}
\begin{remark}\label{rm:DeltaSubscripts} We denote both the coproduct on $\mathbb H^{\symb}$ and that on $\overline{\mathbb H}^{\symb}$ by $\Delta$. When distinction is needed we write $\Delta_{\mathbb H}$, respectively $\Delta_{\overline{\mathbb H}}$.
\end{remark}
\begin{example}\label{ex:CoprodNoInv31}
From Example \ref{ex:Coprod31} we simply apply $\INV$ to the single inverted term.
\begin{equation}
\begin{aligned}
\Delta[y_1,y_2]_{3,1} = \big([y_1,y_2]_{3,1}\otimes 1  + [y_1,y_2]_{2,1}\otimes [y_1]_0 + [y_1,y_2]_{1,1}\otimes \frac{1}{2}[y_1]_0^2\big)\\
+\Big([y_1y_2]_3\otimes[y_2]_1+[y_1y_2]_2\otimes(-[y_2]_2+[y_2]_1[y_1y_2]_0)+\\
\hspace*{3pc}[y_1y_2]_1\otimes ([y_2]_3-[y_2]_2[y_1y_2]_0+\frac{1}{2}[y_2]_1[y_1y_2]_0^2)\Big)\\
-[y_1y_2]_1\otimes ([y_1]_3+\frac{1}{6}[y_1]_0^3)+[y_2]_1\otimes [y_1]_3+ 1\otimes[y_1,y_2]_{3,1} 
\end{aligned}
\end{equation}
\end{example}

\begin{remark}
Note that the only source of denominators is the exponential map, so it follows that the Lie coalgebra of indecomposables is defined over $\Z$.
\end{remark}

\subsection{Associated Hopf algebras}\label{sec:ContractionHopfAlgebras}
One can use $\mathbb H^{\symb}$ to produce other Hopf algebras in a natural way. To see this we introduce some terminology. For $\mathbf i = (i_1,\dots,i_{d+1})$ a strictly increasing sequence with $i_{d+1}\leq n+1$ define 
\begin{equation}\label{eq:Contraction}
    \mathbf i(x_1,\dots,x_n)=
    (x_{i_1\rightarrow i_2},\dots,x_{i_d\rightarrow i_{d+1}}).
\end{equation}
We thus think of $\mathbf i$ as defining a \emph{contraction} from depth $n$ to depth $d$. Note that if $\mathbf j = (j_1,\dots,j_{f+1})$ defines contraction from depth $d$ to depth $f$ the sequence
\begin{equation}
    \mathbf i|\mathbf j = ( i_{j_1},\dots,i_{j_{f+1}}),
\end{equation}
defines a contraction from depth $n$ to depth $f$, and we have
\begin{equation}
    \mathbf j(\mathbf i(x_1,\dots,x_n))=(\mathbf i|\mathbf j)(x_1,\dots,x_n).
\end{equation}

We shall formalize this property of contractions below.
\begin{definition}\label{def:ContractionSystem} A \emph{contraction system} is a collection of sets $X_1$, $X_2$, $\dots$ together with contraction maps $\mathbf i\colon X_n\to X_d$ (one for each $\mathbf i$ as above) satisfying $\mathbf j\circ\mathbf i=\mathbf i\vert\mathbf j\colon X_n\to X_f$.
\end{definition}
Note that if $X=(x_1,\dots,x_n)$ we have
\begin{equation}
[(i_1,i_2)X]_0+[(i_2,i_3)X]_0=[(i_1,i_3)X]_0\in\mathbb H^{\symb}.
\end{equation}

\begin{definition}
For any contraction system $X$ the \emph{free contraction algebra}, $\mathbb{H}(X)$, is generated by symbols $[\alpha]_{n_1,\dots,n_d}$ with $\alpha\in X_d$ and $[\alpha]_0$ with $\alpha\in X_1$ modulo the relation
\begin{equation}\label{eq:ZeroRel}
    [(i_1,i_2)\alpha]_0+[(i_2,i_3)\alpha]_0=[(i_1,i_3)\alpha]_0.
\end{equation}
\end{definition}
For each $\alpha \in X_n$ define an \emph{evaluation at $\alpha$} by
\begin{equation}\label{eq:evalphaDef}
    \ev_\alpha([\mathbf i(x_1,\dots,x_n)]_{\mathbf n})=[\mathbf i(\alpha)]_{\mathbf n}\in \mathbb H(X),
\end{equation}
with $\mathbf n$ denoting either a vector $(n_1,\dots,n_d)$ or $0$. 
\begin{observation}
The coproduct of $[x_1,\dots,x_n]_{k_1,\dots,k_n}\in\mathbb H^{\symb}$ can be expressed entirely in terms of contractions of $(x_1,\dots,x_n)$.
\end{observation}

\begin{theorem}
The free contraction algebra $\mathbb{H}(X)$ has a natural structure as a graded Hopf algebra with coproduct $\Delta_X$ induced by the coproduct on $\mathbb{H}^{\symb}$. Formally,
\begin{equation}
\Delta_X([\alpha]_0)=1\otimes[\alpha]_0+[\alpha]_0\otimes 1,\qquad \Delta_X([\alpha]_{n_1,\dots,n_d})=\ev_\alpha(\Delta([x_1,\dots,x_d]_{n_1,\dots,n_d}).
\end{equation}
\end{theorem}
\begin{proof} We must prove that $\Delta_X$ is coassociative. Since $\mathbf j\mathbf i\alpha=(\mathbf i\vert\mathbf j)\alpha$, we have
\begin{equation}
    \ev_{\mathbf i\alpha}([\mathbf j(x_1,\dots,x_d)]_{\mathbf n})=[\mathbf j\mathbf i\alpha]_{\mathbf n}=[(\mathbf i\vert\mathbf j)\alpha]_{\mathbf n}=\ev_\alpha([(\mathbf i\vert\mathbf j)(x_1,\dots,x_n)]_\mathbf n)=\ev_\alpha([\mathbf j\mathbf i(x_1,\dots,x_n)]_{\mathbf n}).
\end{equation}
Since this holds for all $\mathbf j$ we have
\begin{equation}
    \ev_{\mathbf i\alpha}(\Delta[x_1,\dots,x_d]_\mathbf n)=\ev_\alpha(\Delta[(\mathbf i(x_1,\dots,x_n)]_\mathbf n).
\end{equation}
Hence,
\begin{equation}
    \Delta_X\ev_\alpha([\mathbf i(x_1,\dots,x_n)]_{\mathbf n})=\Delta_X([\mathbf i\alpha]_{\mathbf n})=\ev_{\mathbf i\alpha}(\Delta([x_1,\dots,x_d]_{\mathbf n}))=\ev_\alpha\Delta([\mathbf i(x_1,\dots,x_n)]_{\mathbf n}).
\end{equation}
Since this holds for all $\mathbf i$, we have $\Delta_X\circ\ev_\alpha=\ev_{\alpha}\circ\Delta$. From this we see that coassociativity of $\Delta$ implies that of $\Delta_X$.
\end{proof}

\begin{example}\label{eq:FieldEx}
Let $F$ be a field and define 
\begin{equation}
    X_k=\left\{(x_1,\dots,x_k)\in F^k\bigm\vert x_i\neq 0,\, \prod_{r=i}^j x_r\neq 1, \text{ for all } 1\leq i\leq j\leq k \right\}.
\end{equation}
This is a contraction system with contraction maps defined as in~\eqref{eq:Contraction}. The corresponding Hopf algebra is denoted $\mathbb H^{\symb}(F)$.
\end{example}
\begin{example}\label{eq:FieldWithExtEx}
Let $F$ be a field and let $\pi\colon E\to F^*$ be a torsion free extension of $F^*$ by $\Z$. Define 
\begin{equation}
    X_k=\left\{(u_i,v_{i,j})\in E^{k+\binom{k+1}{2}}\bigm\vert \pi(\sum_{r=i}^j u_i)+\pi(v_{ij})=1 \text{ for all } 1\leq i\leq j\leq k\right\}
\end{equation}
and contraction maps given by
\begin{equation}\label{eq:UpstairsContractions}
    \mathbf i^*(u_s)=\sum_{r=i_s}^{i_{s+1}-1}u_r,\qquad \mathbf i^*(v_{r,s})=v_{i_r,i_{s+1}-1}.
\end{equation}
In particular, if $F=\C$ and $\pi\colon\C\to\C^*$ is the exponential map, we have contraction maps
\begin{equation}\label{eq:ContractionSdhat}
    \mathbf i\colon\widehat S_n(\C)\to\widehat S_d(\C).
\end{equation}
The corresponding Hopf algebra is denoted $\widehat{\mathbb H}_{E}^{\symb}(F)$. 
\end{example}

\begin{example}\label{ex:ComplexManifold}
Let $M$ be a smooth complex manifold and let $U$ be an open subset of $M$. If we let 
$X_k=\Omega^0(U,\widehat S_k(\C))$,
the set of holomorphic maps from $U$ to $\widehat S_k(\C)$, with contraction maps induced by~\eqref{eq:ContractionSdhat}, we obtain a sheaf $\widehat{\mathbb H}^{\symb}_M$ of Hopf algebras on $M$. 
\end{example}

\begin{example} A semisimplicial set defines a contraction system with $\mathbf i$ being induced by the morphism $[d]\to[n]$ in the simplex category taking $k$ to $i_{k+1}-1$. 
\end{example}

\subsection{The variation matrix}\label{sec:symbolicVariationMatrix}
Let $\Z^{\infty}_{\geq 0}$ be the union of all $\Z^{\ell}_{\geq 0}$ with all zero vectors identified. The polylog generators $[x_{i_1\to i_2},\dots,x_{i_d\to i_{d+1}}]_{n_1,\dots,n_d}$ together with 1 are in natural one-one-correspondence with $\Z^{\infty}_{\geq 0}$. Namely, $[x_{i_1\to i_2},\dots,x_{i_d\to i_{d+1}}]_{n_1,\dots,n_d}$ corresponds to
\begin{equation}
    (0^{i_1-1},n_1,0^{i_2-i_1-1},n_2,\dots,0^{i_d-i_{d-1}-1},n_d,0^{i_{d+1}-i_d-1})
\end{equation}
and $1$ corresponds to 0. We shall thus when convenient identify a vector with its corresponding generator and vice versa. For $\mathbf k=(k_1,\dots,k_\ell)\in\Z^\infty_{\geq 0}$, let $\| \mathbf k\|=\sum k_i$ and for $\mathbf k\neq 0$ let $\dim(\mathbf k)=\ell$.
\begin{definition}\label{def:ZinfOrder}
We endow $\Z^{\infty}_{\geq 0}$ with the ordering defined by $\mathbf k\prec\mathbf l$ if
\begin{itemize}
    \item $\|\mathbf k\|<\|\mathbf l\|$
    \item or if $\|\mathbf k\|=\|\mathbf l\|$ and $\dim(\mathbf k)<\dim(\mathbf l)$
    \item or if $\|\mathbf k\|=\|\mathbf l\|$ and $\dim(\mathbf k)=\dim(\mathbf l)$ and the rightmost nonzero entry of $\mathbf l-\mathbf k$ is negative.
\end{itemize}
\end{definition}
\begin{example}
\begin{equation}0\prec (0,1)\prec (1,0)\prec (0,2)\prec(1,1)\prec(2,0)\prec(0,3)\prec(1,2)
\end{equation}
corresponds to
\begin{equation}
    1\prec\Li_1(x_2)\prec\Li_1(x_1x_2)\prec\Li_2(x_2)\prec\Li_{1,1}(x_1,x_2)\prec\Li_2(x_1x_2)\prec\Li_3(x_2)\prec\Li_{1,2}(x_1,x_2).
\end{equation}
\end{example}
\begin{definition}\label{def:variation matrix} The \emph{variation matrix} is the matrix $V$ with rows and columns parameterized by $\Z^{\infty}_{\geq 0}$ defined by
\begin{equation}
  \Delta(v)=\sum_{w\in \Z^{\infty}_{\geq 0}} w\otimes V_{v,w}
\end{equation}
\end{definition}
Note that $V_{v,w}=0$ if $\dim(v)\neq\dim(w)$ or if $v\prec w$. Hence, there are only finitely many entries in each row and column. 

\begin{example} By Example~\ref{ex:CoprodNoInv31} we see that the non-zero entries in the row corresponding to $[x_1,x_2]_{3,1}$ are 
\begin{multline}
     [x_1,x_2]_{3,1},\quad [x_1]_3,\quad [x_2]_3-[x_1]_3-[x_2]_2[x_1x_2]_0+\frac{1}{2}[x_2]_1[x_1x_2]_0^2-\frac{1}{6}[x_1]_0^3,\\\frac{1}{2}[x_1]_0^2,\quad -[x_2]_2+[x_2]_1[x_1x_2]_0,\quad[x_1]_0,\quad[x_2]_1,\quad 1
\end{multline}
corresponding to the columns $1$, $[x_2]_1$, $[x_1x_2]_1$, $[x_1,x_2]_{1,1}$, $[x_1x_2]_2$, $[x_1,x_2]_{2,1}$, $[x_1x_2]_3$, and $[x_1,x_2]_{3,1}$.
\end{example}

The coproduct of a matrix is defined entrywise, and the tensor product is defined by the usual formula for matrix multiplication. The following result is crucial.
\begin{theorem}
The variation matrix satisfies $\Delta(V^T)=V^T\otimes V^T$. 
\end{theorem}
\begin{proof}
Using coassociativity of the coproduct one easily checks that $\Delta V_{w,v}=\sum_u V_{u,v}\otimes V_{w,u}$.
\end{proof}

\begin{definition}\label{def:FiniteV} Let $\mathbf n=(n_1,\dots,n_d)$, where $n_i\geq 1$. The variation matrix for $\Li_\mathbf n$ is the submatrix $V_{\mathbf n}$ of $V$ parametrized by the tuples $\mathbf v\in\Z^{\infty}_{\geq 0}$ with $\mathbf v\preceq \mathbf n$, $\dim(\mathbf v)=d$, and $v_i\leq n_i$.
\end{definition}

\begin{remark}
We may regard $V_{\mathbf n}$ either as a matrix with entries in $\mathbb H^{\symb}$ or as a matrix of multivalued functions on $S_d(\C)$.
\end{remark}

\begin{example}\label{ex:Variation21}
The variation matrix for $\Li_{2,1}(x_1,x_2)$ is 
\begin{equation}
\begin{bmatrix}
1 & 0 & 0 & 0 & 0 & 0 \\
[x_2]_1 & 1 & 0 & 0 & 0 & 0\\
[x_1x_2]_1 & 0 & 1 & 0 & 0 & 0\\
[x_1,x_2]_{1,1} & [x_1]_1 & -[x_1]_1-[x_1]_0+[x_2]_1 & 1 & 0 & 0\\
[x_1x_2]_{2} & 0 & [x_1x_2]_0 & 0 & 1 & 0\\
[x_1,x_2]_{2,1} & [x_1]_2 & -[x_1]_2-\frac{1}{2}[x_1]_0^2-[x_2]_2 + [x_1x_2]_0[x_2]_1 & [x_1]_0 & [x_2]_1 & 1\\
\end{bmatrix}
\end{equation}
\end{example}

\begin{example}
The variation matrix for $\Li_{n}(x_1)$ is given by
\begin{equation}
\begin{bmatrix}
1\\
[x_1]_1&1\\
[x_1]_2&\phantom{\frac{1}{24}}[x_1]_0&1\\
[x_1]_3&\frac{1}{2}[x_1]_0^2&\phantom{\frac{1}{6}}[x_1]_0&1\\
[x_1]_4&\frac{1}{6}[x_1]_0^3&\frac{1}{2}[x_1]_0^2&\phantom{\frac{1}{2}}[x_1]_0&1\\
[x_1]_5&\frac{1}{24}[x_1]_0^4&\frac{1}{6}[x_1]_0^3&\frac{1}{2}[x_1]_0^2&[x_1]_0&1\\
\vdots&\vdots&\vdots&\vdots&\vdots&\vdots&\ddots
\end{bmatrix}
\end{equation}
\end{example}

\begin{remark}
One can also define variation matrices with entries in $\overline{\mathbb H}^{\symb}$ using the coproduct on $\overline{\mathbb H}^{\symb}$ instead. A superscript $\mathbb H$ or $\overline{\mathbb H}$ may be added for distinction (this is needed in Section~\ref{sec:variationOfHodgeStructure}). Each variant of $V$ (either $V^{\mathbb H}$, $V^{\overline{\mathbb H}}$, $V^{\mathbb H}_{\mathbf n}$, or $V^{\overline{\mathbb H}}_{\mathbf n}$) satisfies $\Delta(V^T)=V^T\otimes V^T$.
\end{remark}

\subsection{The antipode}\label{sec:Antipode}
For a connected Hopf algebra with product $\mu$, coproduct $\Delta$, unit $\eta$, and counit $\varepsilon$, the antipode $S$ is uniquely determined by $\mu(S\otimes\id)\circ\Delta=\eta\circ\varepsilon$.
\begin{proposition}
The antipode satisfies $S(V)=V^{-1}$.
\end{proposition}
\begin{proof}
Let $I$ denote the identity matrix. Since $\Delta(V^T)=V^T\otimes V^T$ we have
\begin{equation}
    I=\eta\varepsilon(V^T)=\mu(S\otimes\id)\Delta(V^T)=S(V^T)V^T.
\end{equation}
The result follows.
\end{proof}

\subsection{Derivatives}

 Using the generating series~\eqref{eq:PowerSeriesLi} for $\Li_{\mathbf n}(\mathbf x)$
 one easily computes the partial derivatives $\partial_k=\frac{\partial}{\partial x_k}$ and thus its total derivative $\dX \Li_{\mathbf n}(\mathbf x)$. The result can be conveniently expressed using the generating series
\begin{equation}
    \Li(x_1,\dots,x_d| t_1,\dots,t_d)=\sum_{n_i\geq 1}\Li_{n_1,\dots,n_d}(x_1,\dots,x_d)t_1^{n_1-1}\cdots t_d^{n_d-1}.
\end{equation}
Namely, we have
\begin{equation}
    \begin{aligned}
\dX &\Li(x_1,\dots,x_d|t_1,\dots,t_d)=\Li(x_1,\dots,x_d|t_1,\dots,t_d)\left(\sum_{k=1}^d\dX\log(x_k)t_k\right)\\
&-\Li(x_2,\dots,x_d|t_2,\dots,t_d)~\dX\log(1-x_1)\\
&-\sum_{k=2}^d\Li(x_1,\dots,x_{k-1}x_k,\dots,x_d|t_1,\dots,t_{k-1},t_{k+1},\dots,t_d)~\dX\log(1-x_k)\\
&-\sum_{k=1}^{d-1}\Li(x_1,\dots,x_kx_{k+1},\dots,x_d|t_1,\dots,t_{k-1},t_{k+1},\dots,t_d)\big(-\dX\log(1-x_k)+\dX\log(x_k)\big).
    \end{aligned}
\end{equation}
Motivated by this we now define a derivation on $\mathbb H^{\symb}$.
Let $\dX\mathbb H^{\symb}$ be the free $\mathbb H^{\symb}$ module generated by symbols $\dX [\prod_{r=i}^jx_r]_1$ and $\dX[x_i]_0$ and define a derivation $\dX\colon\mathbb H^{\symb}\to \dX\mathbb H^{\symb}$ by
\begin{equation}\label{eq:dMultLog}
  \begin{aligned}
\dX[y_1,\dots,y_d|&t_1,\dots,t_d]=[y_1,\dots,y_d|t_1,\dots,t_d]\left(\sum_{k=1}^k\dX [y_k]_0t_k\right)\\
&+[y_2,\dots,y_d|t_2,\dots,t_d]\dX[x_1]_1\\&+\sum_{k=2}^d[y_1,\dots,y_{k-1}y_k,\dots,y_d|t_1,\dots,t_{k-1},t_{k+1},\dots,t_d)\dX[y_k]_1\\
&-\sum_{k=1}^{d-1}[]y_1,\dots,y_ky_{k+1},\dots,y_d|t_1,\dots,t_{k-1},t_{k+1},\dots,t_d]\big(\dX[y_k]_1+\dX[y_k]_0\big)
    \end{aligned} 
\end{equation}
and imposing the Leibniz rule on products.

Let $\mathbb H^{\symb}_d$ denote the weight $d$ subspace and let $\Delta_{k,l}$ denote the composition of $\Delta$ with projection onto $\mathbb H^{\symb}_k\otimes\mathbb H^{\symb}_l$.
\begin{lemma}\label{lemma:dAndDeltaPolylog} The restriction of $\dX$ to $\mathbb H^{\symb}_d$ equals $\phi\circ\Delta_{d-1,1}$, where $\phi$ takes $x\otimes y$ to $x~\dX y$.
\end{lemma}
\begin{proof} It is easy to show that $\phi\circ\Delta_{d-1,1}[y_1,\dots,y_d|t_1,\dots,t_d]$ equals the righthand side of~\eqref{eq:dMultLog}. To see that it still holds for products, let $P_1$ and $P_2$ be symbols in weight  $k$ and $l$, respectively. We then have (for $n=k+l$)
\begin{equation}
    \phi\circ\Delta_{n-1,1}(P_1 P_2)=\phi\big(\Delta_{k-1,1}(P_1)(P_2\otimes 1)+(P_1\otimes 1)\Delta_{l-1,1}(P_2)\big)=P_2~\dX P_1+P_1~\dX P_2=\dX (P_1P_2).
\end{equation}
The result follows. 
\end{proof}
\begin{remark} One can similarly define $\dX \overline{\mathbb H}^{\symb}$ and a derivation $\dX\colon\overline{\mathbb H}^{\symb}\to \dX\overline{\mathbb H}^{\symb}$.
The formula is the same except that $\dX[y_k]_1+\dX[y_k]_0$ is replaced by $\dX[y_k^{-1}]_1$.
\end{remark}

\begin{corollary}Let $\omega=\dX V_1$, where $V_1$ is the weight 1 part of $V$.
We have $\dX V=\omega V$. The same holds for $V^{\overline{\mathbb H}}$ and for the submatrices associated to $\mathbf n$.
\end{corollary}
\begin{proof}
    This follows from Lemma~\ref{lemma:dAndDeltaPolylog} and the fact that $\Delta(V^T)=V^T\otimes V^T$.
\end{proof}

\section{The Hopf algebra of tensors and the symbol map}\label{sec:Tensors}
Our main reference for this section is~\cite{DuhrDulat_PolyLogTools}.
\subsection{The tensor algebra}
For an abelian group $A$ let $T^*A$ denote the tensor algebra of $A$. When convenient, we regard elements as words consisting of letters in $A$. It is well known 
that $T^*A$ is a graded Hopf algebra (graded by word length $|\,|$) with coproduct $\Delta$ given by \emph{deconcatenation}
\begin{equation}
\Delta(w)=\sum_{w=w_1w_2}w_1\otimes w_2
\end{equation}
and product given by the \emph{shuffle product} $\shuffle$ defined recursively by
\begin{equation}
(aw_1)\shuffle (bw_2)=a(w_1\shuffle(bw_2))+b((aw_1)\shuffle w_2),\qquad |a|=|b|=1,
\end{equation}
where the recursion starts with $w\shuffle 1=1\shuffle w=w$.

\subsection{The symbol map} For a graded Hopf algebra $H$, let $\Delta_{1,\dots,1}\colon H\to T^* H_1$ denote the maximal iteration of the coproduct. More precisely, 
the restriction of $\Delta_{1,\dots,1}$ to $H_n$ is defined inductively by
\begin{equation}
    \Delta_{1,\dots,1}=(\id\otimes\Delta_{1,\dots,1})\circ\Delta_{1,n-1}=
    (\Delta_{1,\dots,1}\otimes \id)\circ\Delta_{n-1,1}
\end{equation}
where the second equality follows from coassociativity of $\Delta$.

\begin{definition}
We call $\Delta_{1,\dots,1}\colon\mathbb H^{\symb}\to T^*\mathbb H^{\symb}_1$ the \emph{symbol map}, and call the image of $x\in\mathbb H^{\symb}$ the \emph{symbol of $x$}. 
\end{definition}

\begin{example}\label{ex:Symbols}
The symbols of $[x_1]_n$, $[x_1,x_2]_{1,1}$, and $[x_1,x_2]_{2,1}$ are given by $(n\geq 2)$
\begin{equation}
    \begin{gathered}
        -[x_1]_0^{\otimes(n-1)}\otimes[x_1]_1,\qquad ([x_1]_0-[x_1]_1+[x_2]_1)\otimes [x_1x_2]_1+[x_1]_1\otimes[x_2]_1,
\\
[x_1]_0\otimes ([x_1]_0-[x_1]_1+[x_2]_1)\otimes [x_1x_2]_1+[x_2]_1\otimes [x_1x_2]_0\otimes[x_1x_2]_1+[x_1]_0\otimes [x_1]_1\otimes[x_2]_1
    \end{gathered}
\end{equation}
\end{example}

The following result is probably well known, but we are not aware of a reference with a proof.
\begin{proposition}\label{prop:HopfAlgMorphism} $\Delta_{1,\dots,1}$ is a morphism of graded Hopf algebras.
\end{proposition}
\begin{proof}
We only prove that $\Delta_{1,\dots,1}$ preserves the product, since this is all we need. The proof that $\Delta_{1,\dots,1}$ preserves the coproduct is simpler and left to the reader. We must show that
\begin{equation}\label{eq:symbol of product is the shuffle product of symbols}
\Delta_{1,\dots,1}(ab)=\Delta_{1,\dots,1}(a)\shuffle\Delta_{1,\dots,1}(b).
\end{equation}
Suppose $|a|=k$, and $|b|=l$. Letting
\begin{equation}
\Delta_{1,k-1}(a)=\sum_i\alpha_i\otimes a_i,\qquad \Delta_{1,l-1}(b)=\sum_j\beta_j\otimes b_j,
\end{equation}
we have
\begin{equation}
\Delta_{1,k+l-1}(ab)=(\sum_i\alpha_i\otimes a_i)(1\otimes b)+(1\otimes a)(\sum_j\beta_j\otimes b_j)=\sum_i\alpha_i\otimes(a_ib)+\sum_j\beta_j\otimes(ab_j).
\end{equation}
By the inductive definition of $\Delta_{1,\dots,1}$ we have
\begin{equation}
\Delta_{1,\dots,1}(a)=\sum_i\alpha_i\otimes\Delta_{1,\dots,1}(a_i),\qquad \Delta_{1,\dots,1}(b)=\sum_j\beta_j\otimes\Delta_{1,\dots,1}(b_j).
\end{equation}
By induction on weight we have
\begin{equation}
\begin{aligned}
\Delta_{1,\dots,1}(ab)&=(\id\otimes\Delta_{1,\dots,1})\circ\Delta_{1,k+l-1}(ab)\\
&=\sum_i\alpha_i\otimes\Delta_{1,\dots,1}(a_ib)+\sum_j\beta_j\otimes\Delta_{1,\dots,1}(ab_j)\\
&=\sum_i\alpha_i\otimes(\Delta_{1,\dots,1}(a_i)\shuffle\Delta_{1,\dots,1}(b))+\sum_j\beta_j\otimes(\Delta_{1,\dots,1}(a)\shuffle\Delta_{1,\dots,1}(b_j)).
\end{aligned}
\end{equation}
On the other hand, we have
\begin{equation}
\begin{aligned}
&\Delta_{1,\dots,1}(a)\shuffle\Delta_{1,\dots,1}(b)\\
&=\left(\sum_i\alpha_i\otimes\Delta_{1,\dots,1}(a_i)\right)\shuffle\left(\sum_j\beta_j\otimes\Delta_{1,\dots,1}(b_j)\right)\\
&=\sum_{i,j}(\alpha_i\otimes\Delta_{1,\dots,1}(a_i))\shuffle(\beta_j\otimes\Delta_{1,\dots,1}(b_j))\\
&=\sum_{i,j}\big(\alpha_i\otimes(\Delta_{1,\dots,1}(a_i)\shuffle(\beta_j\otimes\Delta_{1,\dots,1}(b_j))\big)+\beta_j\otimes((\alpha_i\otimes\Delta_{1,\dots,1}(a_i))\shuffle\Delta_{1,\dots,1}(b_j))\\
&=\sum_{i}\alpha_i\otimes(\Delta_{1,\dots,1}(a_i)\shuffle(\sum_j\beta_j\otimes\Delta_{1,\dots,1}(b_j)))+\sum_j\beta_j\otimes((\sum_i\alpha_i\otimes\Delta_{1,\dots,1}(a_i))\shuffle\Delta_{1,\dots,1}(b_j))\\
&=\sum_{i}\alpha_i\otimes(\Delta_{1,\dots,1}(a_i)\shuffle\Delta_{1,\dots,1}(b))+\sum_j\beta_j\otimes(\Delta_{1,\dots,1}(a)\shuffle\Delta_{1,\dots,1}(b_j)).
\end{aligned}
\end{equation}
This shows that $\Delta_{1,\dots,1}(ab)=\Delta_{1,\dots,1}(a)\shuffle\Delta_{1,\dots,1}(b)$ as desired.
\end{proof}

\subsection{Killing products}
For any graded, connected, commutative Hopf algebra $H$ with product $\mu$, coproduct $\Delta$, and antipode $S$, there is a natural projection map $\Pi\colon H\to H$ whose kernel is $H_{>0}H_{>0}$. Namely, we have (see~\cite[Sec.~2.2]{CharltonDuhrGangl_SingleValuedPolylogs})
\begin{equation}
    \Pi=\id+Y^{-1}\mu(S\otimes Y)\Delta',
\end{equation}
where $\Delta'=\Delta-\id\otimes 1-1\otimes \id$ is the restricted coproduct, and $Y$ multiplies a homogeneous element by its weight. In the case of $T^*A$, we have
\begin{equation}
    \Pi(a_1\otimes\dots\otimes a_n)=a_1\otimes\dots\otimes a_n+\frac{1}{n}\sum_{i=1}^{n-1}(-1)^i(n-i)a_i\otimes\dots\otimes a_1\shuffle a_{i+1}\otimes\dots\otimes a_n.
\end{equation}
It is not difficult to see that this formula is equivalent to the recursive formula
\begin{equation}\label{eq:PiRecursion}
\Pi(a_1\otimes\dots\otimes a_n)=\frac{n-1}{n}\big(\Pi(a_1\otimes\dots\otimes a_{n-1})\otimes a_n-\Pi(a_2\otimes\dots\otimes a_n)\otimes a_1\big).
\end{equation}
The recursion starts with $\Pi(a)=a$.

\section{Forms}\label{sec:Forms}
Recall that $\mathbb H^{\symb}_1$ is generated by $[x_i]_0$ and $[\prod_{r=i}^jx_r]_1$. It will be convenient to change notation and denote $[x_i]_0$ by $u_i$ and $[\prod_{r=i}^jx_r]_1$ by $-v_{i,j}$. We sometimes write $v_i$ instead of $v_{i,i}$. 

Let $\Omega^*$ denote the algebraic de Rham complex for the polynomial ring over $\Q$ generated by the $u_i$ and $v_{i,j}$. Consider the map
\begin{equation}\label{eq:SymbToForm}
w\colon T^*\mathbb H^{\symb}_1\to\Omega^1,\qquad f_1\otimes\cdots\otimes f_n\mapsto\frac{(-1)^{n+1}}{n!}\sum_{1\leq i\leq n}(-1)^{i-1}\binom{n-1}{i-1}f_1\cdots \dX f_i\cdots f_n.
\end{equation}
\begin{example} We have
\begin{equation}
w(f_1\otimes f_2)=\frac{1}{2}(f_1\dX f_2-f_2\dX f_1),\qquad w(f_1\otimes f_2\otimes f_3)=\frac{1}{6}(f_2f_3\dX f_1-2f_1f_3\dX f_2+f_1f_2\dX f_3)
\end{equation}
\end{example}
\begin{example} By Example~\ref{ex:Symbols}, we have
\begin{equation}
    \begin{gathered}
        w(\Delta_{1,\dots,1}[x_1]_{n})=\frac{(-1)^n}{n!}u_1^{n-2}(u_1\dX v_1-v_1\dX u_1),\quad n\geq 2\\
w(\Delta_{1,\dots,1}[x_1,x_2]_{1,1})=\frac{1}{2}(-u_1 \dX v_{1,2}+ v_{1,2} \dX u_1+ v_1 \dX v_{1,2}- v_2 \dX v_{1,2}-v_{1,2} \dX v_1+ v_{1,2} \dX v_2- v_1 \dX v_2+ v_2 \dX v_1).
    \end{gathered}
\end{equation}
\end{example}
\begin{lemma}\label{lemma:wFac} The map $w$ factors as $w=\eta\circ\Pi$, where $\Pi$ is the map~\eqref{eq:PiRecursion}, and $\eta$ is given by
\begin{equation}
    \eta\colon T^*\mathbb H^{\symb}_1\to\Omega^1,\qquad f_1\otimes\dots\otimes f_n\mapsto \frac{(-1)^{n+1}}{(n-1)!}f_2\dots f_n~\dX f_1.
\end{equation}
\end{lemma}
\begin{proof}
The factorization is obvious in weight 1, so suppose it holds in weight $n-1$.
For notational convenience, let $A_n=\frac{(-1)^{n+1}}{(n-1)!}$ and $B_n=\frac{(-1)^{n+1}}{n!}$ be the weight $n$ coefficients of $\eta$ and $w$ respectively. We then have 
\begin{equation}
    \eta(g_1\otimes\dots\otimes g_n)=\frac{A_n}{A_{n-1}}\eta(g_1\otimes\dots\otimes g_{n-1})g_n.
\end{equation}
Using this, we obtain
\begin{equation}
\begin{aligned}
&\eta\circ\Pi(f_1\otimes\cdots\otimes f_n)\\
&=\frac{n-1}{n}\left(\eta\big(\Pi(f_1\otimes\cdots\otimes f_{n-1})\otimes f_n\big)-\eta\big(\Pi(f_2\otimes\cdots\otimes f_{n})\otimes f_1\big)\right)\\
&=\frac{n-1}{n}\frac{A_n}{A_{n-1}}\left(\eta\circ\Pi(f_1\otimes\cdots\otimes f_{n-1})f_n-\eta\circ\Pi(f_2\otimes\dots\otimes f_n)f_1\right)\\
&=\frac{-B_{n-1}}{n}\left(f_n\sum_{i=1}^{n-1}(-1)^{i-1}\binom{n-2}{i-1}f_1\cdots df_i\cdots f_{n-1}-f_1\sum_{i=2}^{n}(-1)^{i}\binom{n-2}{i-2}f_2\cdots df_i\cdots f_n\right)\\
&=B_n\sum_{i=1}^n(-1)^{i-1}\binom{n-1}{i-1}f_1\cdots df_i\cdots f_n.\\
&=w(f_1\otimes\cdots\otimes f_n).
\end{aligned}
\end{equation}
This concludes the proof.
\end{proof}

By precomposing with $\Delta_{1,\dots,1}$ we obtain maps
\begin{equation}
    \mathbb H^{\symb}\to \Omega^1,\qquad \mathbb L^{\symb}\to\Omega^1,
\end{equation}
which we also denote by $w$. The latter map is well defined by Proposition~\ref{prop:HopfAlgMorphism} and Lemma~\ref{lemma:wFac}. We will later show that it extends to a map of chain complexes $\bigwedge^*\mathbb L^{\symb}\to\Omega^*$. 

We may regard the image of a 1-form in $\mathbb H^{\symb}(k)$ as an element in $\Omega^1(\widehat S_k(\C))$. Recall the contraction maps $\mathbf i\colon \widehat S_n(\C)\to \widehat S_d(\C)$ defined in~\eqref{eq:ContractionSdhat}. 
The following is elementary.
\begin{lemma} If $\mathbf i$ defines a contraction from depth $n$ to depth $d$, we have
\begin{equation}
    w[\mathbf i(x_1,\dots,x_n)]_{\mathbf n}=\mathbf i^*w[x_1,\dots,x_d]_{\mathbf n}.
\end{equation}
\end{lemma}
\begin{example}
$w[x_1x_2]_2=\frac{1}{2}(u_1u_2\dX v_{1,2}-v_{1,2}(\dX u_1+\dX u_2))$.
\end{example}

\subsection{Forms and the variation matrix}
Recall the variation matrix $V$ defined in Section~\ref{sec:symbolicVariationMatrix}, which satisfies $\Delta(V^T)=V^T\otimes V^T$.
For any matrix with entries in a graded ring we use a subscript $n$ to indicate the weight $n$ part. Let 
\begin{equation}
    \Omega=V_1,\qquad \omega=\dX\Omega.
\end{equation}
We warn the reader that $\omega$ denotes a matrix of $1$ forms and $w$ the map from $\mathbb H^{\symb}$ to $\Omega^1$.
\begin{remark} The matrices $V$, $\omega$, and $\Omega$ are infinite dimensional, but as in Definition~\ref{def:FiniteV} a vector $(n_1,\dots,n_d)$, or equivalently a polylogarithm $\Li_{n_1,\dots,n_d}(x_1,\dots,x_d)$ determines a finite submatrix.
\end{remark}
\begin{example}
For $\Li_{2,1}(x_1,x_2)$, we have
\begin{equation}
\Omega=
\begin{bmatrix}
0 & 0 & 0 & 0 & 0 & 0 \\
-v_2 & 0 & 0 & 0 & 0 & 0 \\
-v_{1,2} & 0 & 0 & 0 & 0 & 0 \\
0 & -v_1 & -u_1+v_1-v_2 & 0 & 0 & 0 \\
0 & 0 & u_{1}+u_2 & 0 & 0 & 0 \\
0 & 0 & 0 & u_1 & -v_2 & 0 \\
\end{bmatrix}
\end{equation}
\end{example}
\begin{lemma} We have
\begin{equation}\label{eq:wVn}
    w(V_n)=\frac{1}{n!}\sum_{k+l=n-1}(-1)^k\binom{n-1}{k}\Omega^k\omega\Omega^l.
\end{equation}
\end{lemma}
\begin{proof}
Since $\Delta(V^T)=V^T\otimes V^T$, it follows that $\Delta_{n-1,1} V_n^T=V_{n-1}^T\otimes \Omega^T$. Therefore,
\begin{equation}
    \Delta_{1,\dots,1}(V^T)=I+\Omega^T+\Omega^T\otimes\Omega^T+\Omega^T\otimes\Omega^T\otimes\Omega^T+\cdots,
\end{equation}
from which it follows that
\begin{equation}
    w(V_n^T)=\frac{(-1)^{n+1}}{n!}\sum_{i=1}^n(-1)^{i-1}\binom{n-1}{i-1}(\Omega^T)^{i-1}\omega^T(\Omega^T)^{n-i}.
\end{equation}
Hence,
\begin{equation}
    w(V_n)=\frac{(-1)^{n+1}}{n!}\sum_{i=1}^n(-1)^{i-1}\binom{n-1}{i-1}\Omega^{n-i}\omega\Omega^{i-1}.
\end{equation}
The result follows after reindexing $k=n-i$.
\end{proof}
\subsection{Lifted polylogarithms and connection matrices}
We now introduce a matrix of 1-forms $\widehat w$ and a matrix of lifted multiple polylogarithms $\widehat V$, which will play a crucial role in our construction of the chain map $\bigwedge^*\mathbb L^{\symb}\to\Omega^*$, and also in the lift of the variation of mixed Hodge structure for polylogarithms (Section~\ref{sec:variationOfHodgeStructure}). Let
\begin{equation}\label{eq:dOmegahatDef}
    \widehat \omega=\dX e^{-\Omega}e^\Omega+e^{-\Omega}\omega e^\Omega,\qquad \widehat V=e^{-\Omega}V.
\end{equation}
\begin{lemma}\label{lemma:domegahat} We have $\dX\widehat\omega-\widehat\omega\wedge \widehat \omega=-e^{-\Omega}\omega\wedge\omega e^{\Omega}$ and $\dX\widehat V=\widehat\omega\widehat V$.
\end{lemma}
\begin{proof} This follows directly from the definition using basic calculus of forms.
\end{proof}
\begin{remark}
We show later that when regarded as forms on $\widehat S_d(\C)$, we have $\omega\wedge\omega=0$, so $\dX\widehat\omega=\widehat\omega\wedge \widehat \omega$. This will play a crucial role in Section~\ref{sec:variationOfHodgeStructure}.
\end{remark}

\begin{example}
In depth 1, we have explicit expressions of $\Omega$, $\widehat\omega$, and $\widehat V$ as follows
\begin{equation}
\Omega=\begin{bmatrix}
0\\
-v_1&0\\
&u_1&0\\
&&u_1&0\\
&&&\ddots&\ddots\\
\end{bmatrix},\quad
\widehat\omega=\begin{bmatrix}
0&0&\cdots\\
0&0&\cdots\\
w_2(x_1)&0&\cdots\\
2w_3(x_1)&0&\cdots\\
3w_4(x_1)&0&\cdots\\
4w_5(x_1)&0&\cdots\\
\vdots&\vdots&\ddots\\
\end{bmatrix},\quad
\widehat V=\begin{bmatrix}
0&0&\cdots\\
0&0&\cdots\\
\widehat\L_2(x_1)&0&\cdots\\
\widehat\L_3(x_1)&0&\cdots\\
\widehat\L_4(x_1)&0&\cdots\\
\widehat\L_5(x_1)&0&\cdots\\
\vdots&\vdots&\ddots\\
\end{bmatrix}
\end{equation}
where 
\begin{equation}
    \widehat\L_n(x_1)=-\frac{(-1)^n}{n!}u_1^{n-1}v_1+\sum_{r=0}^{n-1}\frac{(-1)^r}{r!}[x_1]_{n-r}u_1^r.
\end{equation}
In~\cite{Zickert_HolomorphicPolylogarithmsAndBlochComplexes} it was shown that $\widehat{\L}_n$ gives rise to a single valued function on $\widehat{S}_1$ modulo $\frac{(2\pi i)^n}{(n-1)!}$.
\end{example}

\begin{example} 
Using the variation matrix for $[x_1,x_2]_{2,1}$ in Example ~\ref{ex:Variation21}, we obtain
\begin{equation}
\widehat\omega=\begin{bmatrix}
\mathbf 0&\mathbf 0\\A&\mathbf0
\end{bmatrix},\qquad\widehat V=\begin{bmatrix}
\mathbf I&\mathbf 0\\B&\mathbf I
\end{bmatrix},
\end{equation}
where $\mathbf I$ and $\mathbf 0$ are $3\times 3$ identity and zero matrices, and $A$ and $B$ are given by
\begin{equation}
\begin{bmatrix}
w[x_1,x_2]_{1,1} &  0 & 0 \\
w[x_1x_2]_ 2 & 0  & 0 \\
2w[x_1,x_2]_ {2, 1} & w[x_1]_ 2 & - w[x_1]_2 - w[x_2]_ 2 
\end{bmatrix},\quad \begin{bmatrix}
\widehat{\mathcal L}_{1, 1} (x_ 1, x_ 2) & 0 & 0 \\
\widehat{\mathcal L}_ 2 (x_1x_2) & 0 & 0 \\
\widehat{\mathcal L}_{2,1}(x_1,x_2) & \widehat{\mathcal L}_ 2 (x_ 1) & - \widehat{\mathcal L}_ 2 (x_ 1) - \widehat{\mathcal L}_ 2 (x_ 2)
\end{bmatrix},
\end{equation}
and
\begin{equation}
\begin{aligned}
 \widehat{\mathcal L}_{1,1}(x_1,x_2) =& \frac{1}{2}\left(-u_1 v_{1,2} + v_1 v_{1,2} - v_1 v_2 - v_{1,2} v_2\right)+ [x_1,x_2]_{1,1}\\
   \widehat{\mathcal L}_{2,1}(x_1,x_2)=& \frac{1}{3} \left(u_1^2 v_{1,2} - u_1 v_1 v_{1,2} + u_1 v_1 v_2 + 2 u_1 v_{1,2} v_2 + u_2 v_{1,2} v_2\right)\\
   &\hspace*{1pc}+ v_2 [x_1x_2]_2-u_1[x_1,x_2]_{1,1}+[x_1,x_2]_{2,1}. \\
\end{aligned}
\end{equation}
\end{example}

The following result gives a concrete formula for $\widehat\omega$.
\begin{lemma}
Let $\Omega$ be any nilpotent matrix with entries in $\Q[\{u_i\},\{v_{jk}\}]$ and let $\omega=\dX \Omega$. We have
\begin{equation}\label{eq:OmegaHatFormula}
\widehat{\omega}=\dX e^{-\Omega}e^\Omega+e^{-\Omega}\omega e^\Omega=\sum_{n\geq1}\frac{n-1}{n!}\sum_{k+l=n-1}(-1)^{k}\binom{n-1}{k}\Omega^{k}\omega\Omega^{l}.
\end{equation}
\end{lemma}
\begin{proof} First we compute
\begin{equation}
\begin{aligned}
&\dX e^{-\Omega}e^\Omega+e^{-\Omega}\omega e^\Omega\\
&=\dX\left(\sum_{j\geq0}(-1)^j\frac{\Omega^j}{j!}\right)\left(\sum_{r\geq0}\frac{\Omega^r}{r!}\right)+\left(\sum_{k\geq0}(-1)^k\frac{\Omega^k}{k!}\right)\omega\left(\sum_{l\geq0}\frac{\Omega^l}{l!}\right) \\
&=\left(\sum_{p,q\geq0}\frac{(-1)^{p+q+1}}{(p+q+1)!}\Omega^p\omega\Omega^q\right)\left(\sum_{r\geq0}\frac{\Omega^r}{r!}\right)+\left(\sum_{k\geq0}(-1)^k\frac{\Omega^k}{k!}\right)\omega\left(\sum_{l\geq0}\frac{\Omega^l}{l!}\right) \\
&=\sum_{n\geq1}\left(\sum_{\substack{p+q+r=n-1\\p,q,r\geq0}}\frac{(-1)^{p+q+1}}{(p+q+1)!r!}\Omega^p\omega\Omega^{q+r}+\sum_{\substack{k+l=n-1\\k,l\geq0}}\frac{(-1)^k}{k!l!}\Omega^k\omega\Omega^l\right).
\end{aligned}
\end{equation}
Evaluating the interior sum, we obtain
\begin{equation}
\begin{aligned}
&\sum_{\substack{p+q+r=n-1\\p,q,r\geq0}}\frac{(-1)^{p+q+1}}{(p+q+1)!r!}\Omega^p\omega\Omega^{q+r}+\sum_{\substack{k+l=n-1\\k,l\geq0}}\frac{(-1)^k}{k!l!}\Omega^k\omega\Omega^l \\
&=\sum_{k+l=n-1}\Omega^k\omega\Omega^l\left(\frac{(-1)^k}{k!l!}+\sum_{q+r=l}\frac{(-1)^{k+q+1}}{(k+q+1)!r!}\right) \\
&=\sum_{k+l=n-1}\Omega^k\omega\Omega^l\left(\frac{(-1)^k}{k!l!}+\frac{(-1)^{n}}{n!}\sum_{r=0}^l(-1)^{r}\binom{n}{r}\right) \\
&=\sum_{k+l=n-1}\Omega^k\omega\Omega^l\left(\frac{(-1)^k}{k!l!}+\frac{(-1)^{n}}{n!}(-1)^l\binom{n-1}{l}\right) \\
&=\frac{n-1}{n!}\sum_{k+l=n-1}(-1)^{k}\binom{n-1}{k}\Omega^k\omega\Omega^l.
\end{aligned}
\end{equation}
This proves the result.
\end{proof}

\begin{corollary}
We have $\widehat\omega_n=(n-1)w(V_n)$.
\end{corollary}
\begin{proof}
This follows by comparing~\eqref{eq:wVn} with~\eqref{eq:OmegaHatFormula}.
\end{proof}

\subsection{Proof of Theorem~\ref{thm:ComDiagram}}
We must prove that the diagram
\begin{equation}
    \xymatrixcolsep{5pc}\cxymatrix{{\mathbb L^{\symb}\ar[r]^{\delta}\ar[d]^-w&\wedge^2(\mathbb L^{\symb})\ar^{\delta\wedge\id-\id\wedge\delta}[r]\ar[d]^{w\wedge w}&\wedge^3(\mathbb L^{\symb})\ar[r]\ar[d]^{w\wedge w\wedge w}&\dots\\\Omega^1\ar[r]^-d&\Omega^2\ar[r]^-d&\Omega^3\ar[r]&\dots}}
\end{equation}
commutes. It is enough to prove commutativity of the lefthand square. Since $\delta(V^T)=V^T\wedge V^T$, the result follows from Proposition~\ref{prop:dw=wwedgew} below.
\begin{proposition}\label{prop:dw=wwedgew} We have $\dX w(V^T)=w(V^T)\wedge w(V^T)$.
\end{proposition}
\begin{proof} The result is equivalent to proving that
\begin{equation}
    \dX w(V_n)+\sum_{p+q=n}w(V_p)\wedge w(V_q)=0.
\end{equation}
We have
\begin{equation}
\begin{aligned}
&(n-1)\left(\dX w(V_n)+\sum_{p+q=n}w(V_p)\wedge w(V_q)\right)\\
&=(n-1)\dX w(V_n)+\sum_{p+q=n}(pq-(p-1)(q-1))w(V_p)\wedge w(V_q)\\
&=\dX\widehat\omega_n-\sum_{p+q=n}\widehat\omega_p\wedge\widehat\omega_q+\sum_{p+q=n}pq~w(V_p)\wedge w(V_q)\\
&=-\sum_{p+q=n}(-1)^{p-1}\frac{\Omega^{p-1}}{(p-1)!}\omega\wedge\omega\frac{\Omega^{q-1}}{(q-1)!}+\sum_{p+q=n}pq~w(V_p)\wedge w(V_q),\\
\end{aligned}
\end{equation}
with the last equality following from Lemma~\ref{lemma:domegahat}. But since
\begin{equation}
\begin{aligned}
p~w(V_p)\wedge q~w(V_q)=\sum_{\substack{1\leq r\leq p\\1\leq s\leq q}}\frac{(-1)^{r-1+q-s}}{(p-1)!(q-1)!}\binom{p-1}{r-1}\binom{q-1}{s-1}\Omega^{r-1}\omega\Omega^{p-r}\wedge\Omega^{q-s}\omega\Omega^{s-1},
\end{aligned}
\end{equation}
we have
\begin{equation}
\begin{aligned}
&\sum_{p+q=n}pq~w(V_p)\wedge w(V_q)\\
&=\sum_{p+q=n}\sum_{\substack{1\leq r\leq p\\1\leq s\leq q}}\frac{(-1)^{r-1+q-s}}{(p-1)!(q-1)!}\binom{p-1}{r-1}\binom{q-1}{s-1}\Omega^{r-1}\omega\wedge\Omega^{n-r-s}\omega\Omega^{s-1}\\
&=\sum_{r+s\leq n}\Omega^{r-1}\omega\wedge\left[\sum_{\substack{p+q=n\\p\geq r,q\geq s}}\frac{(-1)^{r-1+q-s}}{(p-1)!(q-1)!}\binom{p-1}{r-1}\binom{q-1}{s-1}\right]\Omega^{n-r-s}\omega\Omega^{s-1}\\
&=\sum_{r+s\leq n}\Omega^{r-1}\omega\wedge\frac{(-1)^{r-1}}{(r-1)!(s-1)!}\left[\sum_{\substack{p+q=n\\p\geq r,q\geq s}}\frac{(-1)^{q-s}}{(p-r)!(q-s)!}\right]\Omega^{n-r-s}\omega\Omega^{s-1}\\
&=\sum_{r+s\leq n}\Omega^{r-1}\omega\wedge\frac{(-1)^{r-1}}{(r-1)!(s-1)!}\left[\sum_{\substack{p+q=n-r-s\\p,q\geq0}}\frac{(-1)^{q}}{p!q!}\right]\Omega^{n-r-s}\omega\Omega^{s-1}\\
&=\sum_{r+s=n}\frac{(-1)^{r-1}}{(r-1)!(s-1)!}\Omega^{r-1}\omega\wedge\omega\Omega^{s-1},
\end{aligned}
\end{equation}
and it follows that $(n-1)\left(\dX w(V_n)+\sum_{p+q=n}w(V_p)\wedge w(V_q)\right)=0$, which proves the result.
\end{proof}

\begin{corollary}
Let $M$ be a complex manifold. The map $w$ induces a map of chain complexes $\bigwedge^*\widehat{\mathbb L}^{\symb}_M(U)\to\Omega^*_M(U)$, where $\widehat{\mathbb L}_M^{\symb}$ is the sheaf of Lie coalgebras coming from Example~\ref{ex:ComplexManifold}.
\end{corollary}

\subsection{Recurrence relations}\label{sec:Recurrence}
The one forms satisfy a simple recurrence relation, which we express using the variation matrix. An explicit combinatorial formula for the recurrence was described in \cite{ZackThesis}.

\begin{proposition}\label{cor:RecursiveForm}
Let $[A,B]=AB-BA$ be the commutator. For all $n\geq 1$ we have
\begin{equation}
    (n+1)! w(V_{n+1}) = [n!w(V_n),\Omega].
\end{equation}
\end{proposition}
\begin{proof}
This is an elementary computation using~\eqref{eq:wVn}.
\end{proof}
\begin{corollary}\label{cor:Recurrence} We have
$w[x_1,\dots,x_d]_{\mathbf{n}} = \frac{1}{\|\mathbf n\|}\left(w(V_{\|\mathbf n\|-1})_{\mathbf{n},\bullet}\Omega_{\bullet,0}-\Omega_{\mathbf{n},\bullet}w(V_{\|\mathbf n\|-1})_{\bullet,0} \right)$.
\end{corollary}
\begin{proof}
Since $w[x_1,\dots,x_d]_{\mathbf{n}}=w(V_{\|\mathbf n\|})_{\mathbf{n},0}$ the result follows from Proposition~\ref{cor:RecursiveForm}.
\end{proof}
\begin{example}\label{ex:21Recurrence}
We compute the recurrence for $w[x_1,x_2]_{2,1}$. Applying $w$ to the first column and last row of the variation matrix for $\Li_{2,1}(x_1,x_2)$ given in Example \ref{ex:Variation21} we obtain:
\begin{equation}
    \begin{tabular}{r c | c c| c c| c l}
         [& 1 & $-v_2$ & $-v_{1,2}$ & $w[x_1,x_2]_{1,1}$ & $w[x_{1\to 2}]_2$ & $w[x_1,x_2]_{2,1}$&$]^T$\\
         $[$&$w[x_1,x_2]_{2,1}$ & $w[x_1]_{2}$ & $-w[x_1]_2-w[x_2]_2$  &$u_1$ & $-v_2$ & $1$ & $]$\\
    \end{tabular}
\end{equation}
It follows that we have
\begin{equation}
\begin{aligned}
    w[x_1,x_2]_{2,1} =& \frac{1}{3}\big(-v_2 w[x_1]_2 - v_{1,2}(-w[x_1]_2-w[x_2]_2)-\left(u_1 w[x_1,x_2]_{1,1} - v_2  w[x_{1\to 2}]_2\right) \big).
\end{aligned}
\end{equation}
\end{example}

\section{Variations of mixed Hodge-Tate structures}\label{sec:variationOfHodgeStructure}
For fixed $\mathbf n$ let $V^{\overline{\mathbb H}}$, resp.~$V^{\mathbb H}$, denote the variation matrix associated to $\Li_{\mathbf n}(x_1,\dots,x_d)$ with coefficients in $\overline{\mathbb H}^{\symb}$, resp.~$\mathbb H^{\symb}$. We omit the superscripts when the coefficients are clear from context or irrelevant. In the following we regard $V$ as a matrix of multivalued functions on $S_d(\C)$ satisfying the differential equation $\dX V=\omega V$.

For example, for $\mathbf n=(1,1)$, $V^{\overline{\mathbb H}}$ is given in~\eqref{eq:VarMatrices}, whereas $V^{\mathbb H}$ equals
\begin{equation}
V^{\mathbb H}=\begin{bmatrix}
1&0&0&0\\
\Li_1(x_2)&1&0&0\\
\Li_1(x_1x_2)&0&1&0\\
\Li_{1,1}(x_1,x_2)&\Li_1(x_1)&\Li_1(x_2)-\Li_1(x_1)-\log(x_1)&1\\
\end{bmatrix}.
\end{equation}

\begin{lemma}\label{lemma:Flatness} We have $0=\dX\omega=\omega\wedge\omega$.
\end{lemma}
\begin{proof}
Since all entries of $\omega$ have the form $\dX \log(y)$, $\dX \omega=0$. Since $\dX V=\omega V$, we have
\begin{equation}
    0=\dX^2V=\dX(\omega V)=\dX\omega\wedge V-\omega\wedge\dX V=(\dX\omega-\omega\wedge\omega)\wedge V.
\end{equation}
The result follows.
\end{proof}
Let $n=\|\mathbf n\|$ and let $N$ be such that $V$ is $N\times N$.
\subsection{Division into weight blocks}
There are unique indices $\{\mu_p\}_{p=0}^n$ such that the first $\mu_p$ rows and last $\mu_p$ columns are those of weight $\leq p$. The indices divide $V$ into blocks, with the $(k, l)$ block being the submatrix consisting of entries $V_{ij}$ such that $\mu_{k-1}<i\leq \mu_k$, and  $\mu_{l-1}<j\leq \mu_l$. Note that the entries in the $(k,l)$ block have weight $k-l$. Let $\tau(2\pi i)$ be the diagonal matrix whose $k$th block consists of $(2\pi i)^{k}$. 

\begin{example}
The variation matrix of $\Li_{1,1}(x_1,x_2)$ and $\tau(2\pi i)$ can be divided as follows
\begin{equation}
V=\left[
\begin{array}{c|cc|c}
1&0&0&0\\
\hline
[x_2]_1&1&0&0\\

[x_1x_2]_1&0&1&0\\
\hline
[x_1,x_2]_{1,1}&[x_1]_1&[x_2]_1-[x_1^{-1}]_1&1
\end{array}
\right],\quad\tau(2\pi i)=\left[
\begin{array}{c|cc|c}
1&0&0&0\\
\hline
0&2\pi i&0&0\\

0&0&2\pi i&0\\
\hline
0&0&0&(2\pi i)^2
\end{array}
\right]
\end{equation}
\end{example}

\subsection{Variations of mixed Hodge-Tate structures}
Recall the Riemann-Hilbert correspondence, which states that there is an equivalence of categories between local systems (locally constant sheaves) over a complex manifold $X$ and flat connections over $X$.

\begin{theorem}[\cite{Zhao_MultipleZetaFunctionsMultiplePolylogarithmsAndTheirSpecialValues}]\label{thm:RiemannHilbertOnSd}
$\nabla=\dX-\omega$ defines a flat connection on the trivial vector bundle $S_d(\mathbb C)\times\mathbb C^N\to S_d(\mathbb C)$, and the columns of $V\tau(2\pi i)$ generate the global sections of the local system corresponding to $\nabla$.
\end{theorem}

\begin{proof}
Flatness is equivalent to $\nabla\circ\nabla=0$, which is an immediate consequence of Lemma~\ref{lemma:Flatness}.
\end{proof}

We give a brief review of the basic definitions pertaining to variations of mixed Hodge-Tate structures:
The (rational) \emph{Hodge-Tate structure} $\mathbb Q(n)$ of weight $-2n$ is the unique Hodge structure with only $H^{-n,-n}=(\pi i)^n\mathbb Q$. A \emph{mixed Hodge-Tate structure} is a rational vector space $H$ with a decreasing Hodge filtration $F^\bullet$ of $H\otimes\C$ and an increasing weight filtration $W_\bullet$ which are compatible, i.e. the weight piece $\gr^W_{-2k}(H)$ is a direct sum of $\mathbb Q(k)$. A \emph{variation of mixed Hodge-Tate structure} over a complex manifold $X$ is a locally constant sheaf $\mathcal H$ of $\Q$-vector spaces with a decreasing Hodge filtration $F^\bullet$ of $H\otimes\mathcal O_X$ and an increasing weight filtration $W_\bullet$ such that $(\mathcal H_x,F_x^\bullet,(W_x)_\bullet)$ is a mixed Hodge-Tate structure and $F^\bullet$ satisfies Griffiths transversality $\nabla F^p\subseteq F^{p-1}\otimes\Omega^1_X$.

\begin{theorem}[\cite{Zhao_MultipleZetaFunctionsMultiplePolylogarithmsAndTheirSpecialValues}]\label{thm:HodgeStructureOnSd}
The columns $\{C_j\}_{j=1}^N$ of $V^{\overline{\mathbb H}}\tau(2\pi i)$ define a variation of Hodge-Tate structure over $S_d(\mathbb C)$ as follows: Let $\{e_i\}_{i=1}^N$ denote the standard basis of $\mathbb C^N$. The Hodge filtration and weight filtration are given by
\begin{equation}
F^{-p}=\mathbb C\langle\{e_i\}_{i=1}^{\mu_p}\rangle,\quad W_{1-2m}=W_{-2m}=\mathbb Q\langle\{C_j\}_{j\geq \mu_m}\rangle.
\end{equation}
\end{theorem}

\begin{proof}
The $k$-th graded weight piece $\gr^W_k$ is the $(k,k)$-th block of $V^{\overline{\mathbb H}}\tau(2\pi i)$, which is $(2\pi i)^k$ times the identity matrix, thus evidently a direct sum of Hodge-Tate structures. To ensure that the weight filtration is well-defined under analytic continuation amounts to showing that the monodromy doesn't affect the weight filtration. Zhao~\cite{Zhao_AnalyticContinuationOfMultiplePolylogarithms} gives explicit formulas for the monodromy in the case $\mathbf n=(1,\dots,1)$ and refers to Deligne and Goncharov~\cite{DeligneGoncharov_GroupesFondamentauxMotiviquesDeTateMixte} for an abstract proof of the general case. Finally, Griffith transversality follows from the fact that $dV^{\overline{\mathbb H}}=\omega V^{\overline{\mathbb H}}$, which implies that $\dX C_i=\omega C_i\subseteq\mathbb C\langle\{e_j\}_{j=1}^{\mu_{p-1}}\rangle\otimes\Omega_X^1$ for any $\mu_{p-1}<i\leq\mu_p$.
\end{proof}

\begin{remark}
Explicit formulas for the monodromy for arbitrary $\mathbf n$ have been derived by Haoran Li (PhD thesis to appear).
\end{remark}

\begin{theorem} Theorem~\ref{thm:HodgeStructureOnSd} also holds for $V^{\mathbb H}$, and the two variations of mixed Hodge-Tate structure are identical.
\end{theorem}
\begin{proof} The local system with its Hodge and weight filtration is defined as in Theorem~\ref{thm:HodgeStructureOnSd}. Let $U=V^{\mathbb H}\tau(2\pi i)$ and $W=V^{\overline{\mathbb H}}\tau(2\pi i)$. Monodromy invariance and the fact that the Hodge-Tate structures are identical follow by showing that $U^{-1}V$ is a constant (as a multivalued function, e.g.~$\log(x)-\log(-x)$ is constant) rational matrix. Since $U$ and $W$ both satisfy the differential equation $\dX X=\omega X$, it follows that $U^{-1}W$ is a constant $C$, so we must prove that $C$ is rational.
By Goncharov's inversion formula~\cite[(34)]{Goncharov_MultiplePolylogarithmsAndMixedTateMotives}, each inverted polylogarithm $\Li_{\mathbf k}(\mathbf x^{-1})$ is a rational polynomial in regular polylogarithms $\Li_{\mathbf k}(\mathbf x)$, logarithms $\log(x_i)$ and powers of $\pi i$. The constant term is a rational multiple of $(\pi i)^{|\mathbf k|}$. We can thus canonically write each entry of $U^{-1}W(z,\dots,z)=C$ as $f_0+f_1(z)\log(z)+\dots+f_k(z)\log(z)^k$, where $f_0\in\Q$ and where, for $i>0$, $f_i$ is holomorphic in a neighborhood of $0$ with $f_i(0)=0$. Clearly such an expression can only be constant if all $f_i$ are 0, so $C$ is rational.
\end{proof}

\subsection{Lifted variations of mixed Hodge-Tate structure}
Since $\widehat S_d(\C)$ covers $S_d(\C)$ we may also regard $V=V^{\mathbb H}$ as a matrix of multivalued functions on $\widehat S_d(\C)$. We note that $\Omega=V_1$ is single valued. As in~\eqref{eq:dOmegahatDef} define
\begin{equation}
    \widehat \omega=\dX e^{-\Omega}e^\Omega+e^{-\Omega}\omega e^\Omega,\qquad \widehat V=e^{-\Omega}V,
\end{equation}
where $\omega=d\Omega$. It follows from Lemma~\ref{lemma:domegahat} and Lemma~\ref{lemma:Flatness} that  
\begin{equation}
    \dX\widehat\omega=\widehat\omega\wedge\widehat\omega+e^{-\Omega}\omega\wedge\omega e^\Omega=\widehat\omega\wedge\widehat\omega,\qquad \dX \widehat V=\widehat\omega\widehat V.
\end{equation}

\begin{theorem}\label{thm:LiftedMHS}
$\widehat\nabla=\dX-\widehat\omega$ defines a flat connection on the trivial vector bundle $\widehat S_d(\mathbb C)\times\mathbb C^N\to \widehat S_d(\mathbb C)$, and the columns of $\widehat V\tau(2\pi i)$ generate the global sections of the corresponding local system.
\end{theorem}
\begin{proof}
Since $\dX\widehat\omega=\widehat\omega\wedge\widehat\omega$, $\widehat\nabla$ is flat.
\end{proof}

\begin{theorem}
The columns $\{C_j\}_{j=1}^N$ of $\widehat{V}\tau(2\pi i)$ define a variation of Hodge-Tate structure over $\widehat S_d(\mathbb C)$ with Hodge filtration and weight filtration given by
\begin{equation}
F^{-p}=\mathbb C\langle\{e_i\}_{i=1}^{\mu_p}\rangle,\quad W_{1-2m}=W_{-2m}=\mathbb Q\langle\{C_j\}_{j\geq \mu_m}\rangle
\end{equation}
\end{theorem}

\begin{proof}
Griffith transversality follows from the fact that $\widehat V=\widehat\omega\widehat V$. All else follows from Theorem~\ref{thm:HodgeStructureOnSd}.
\end{proof}

\section{Goncharov's Hopf algebra of iterated integrals}\label{sec:IteratedIntegrals}
For a set $S$ Goncharov~\cite{Goncharov_GaloisSymmetriesOfFundamentalGroupoidsAndNoncommutativeGeometry} defined a graded Hopf algebra $\mathcal I(S)$. It is freely generated by symbols $I(a_0;a_1,\dots,a_d;a_{d+1})$ in weight $d\geq 1$ ($a_i\in S$). The coproduct is given by
\begin{multline}\label{eq:IteratedIntegralCoproduct}
\Delta I(a_0;a_1,\cdots;a_d;a_{d+1})=\\\sum_{0=i_0<i_1<\cdots<i_k<i_{k+1}=d+1}I(a_{i_0};a_{i_1},\cdots,a_{i_k};a_{i_{k+1}})\otimes\prod_{p=0}^kI(a_{i_p};a_{i_p+1},\cdots,a_{i_{p+1}-1};a_{i_{p+1}}),
\end{multline}
where $I(a_0;\emptyset;a_1)=1$. An elementary proof that $\Delta$ is coassociative is given below. By assigning complex numbers to the $a_i$ we can realize $I(a_0;a_1,\dots,a_d;a_{d+1})$ as an iterated integral
\begin{equation}\label{eq:IteratedIntegral}
\int_{\gamma}\frac{dt}{t-a_1}\cdots\frac{dt}{t-a_d},
\end{equation}
where $\gamma$ is a path from $a_0$ to $a_{d+1}$ (there is a canonical renormalization in the case when the integral diverges~\cite{Goncharov_MultiplePolylogarithmsAndMixedTateMotives}).
\begin{remark}
We note that our $\mathcal I(S)$ is Goncharov's $\widetilde{\mathscr I}_\bullet(S)$ (see~\cite[p.~225]{Goncharov_GaloisSymmetriesOfFundamentalGroupoidsAndNoncommutativeGeometry}).
\end{remark}

\subsection{The variation matrix for iterated integrals}\label{sec:VariationMatrices}
Fix a sequence $\mathbf a=\{a_i\}_{i=0}^\infty$ of elements in $S$. We now define a matrix $V=V_{\mathbf a}$ whose (unordered) rows and columns are parameterized by strictly increasing sequences $\mathbf i=(i_0,\dots,i_{n+1})$ of non-negative integers. For such $\mathbf i$ define 
\begin{equation}
I_{\mathbf i}=I(a_{i_0};a_{i_1},\dots,a_{i_n};a_{i_{n+1}}).
\end{equation}
We write $\vert (i_0,\dots,i_{n+1})\vert=n$, and $\mathbf j\leq \mathbf i$ if $\mathbf j$ is a subsequence of $\mathbf i$ with $j_0=i_0$ and $j_{|\mathbf j|+1}=i_{|\mathbf i|+1}$. If so, it follows that for each $p\in\{0,\dots,|\mathbf j|+1\}$ there exists $k_p$ with $i_{k_p}=j_p$. For $p\leq |\mathbf j|$ we denote the $p^{th}$ subsequence of $\mathbf i$, $(i_{k_p},i_{k_p+1},\dots,i_{k_{p+1}})$ by $\mathbf i\cap \mathbf j(p)$. 
With this notation we see that 
\begin{equation}
\Delta I_{\mathbf i}=\sum_{\mathbf j\leq\mathbf i} I_{\mathbf j}\otimes \prod_{p=0}^{|\mathbf j|} I_{\mathbf i\cap \mathbf j(p)}.
\end{equation}
\begin{definition} The \emph{variation matrix} for $\mathbf a$ is the matrix $V$ whose 
$(\mathbf i,\mathbf j)$ entry $V_{\mathbf i,\mathbf j}$ is given by
\begin{equation}
V_{\mathbf i,\mathbf j}=\begin{cases} \prod_{p=0}^{|\mathbf j|} I_{\mathbf i\cap \mathbf j(p)}&\text{if }\mathbf j\leq \mathbf i\\0&\text{otherwise.}\end{cases}
\end{equation}
\end{definition}
\begin{example}
Let $\mathbf{i} = (0,1,3,4,5)$ and $\mathbf j=(0,3,5)$, then $V_{\mathbf i, \mathbf j} = I(a_0;a_1;a_3)I(a_3;a_4;a_4)$. \\
However, for $\mathbf j=(0,2,5)$ we don't have $\mathbf j \leq \mathbf i$ and so $V_{\mathbf i, \mathbf{j}} = 0$.
\end{example}
\begin{example}
If $\mathbf j=\mathbf{i}$, every factor is of the form $I(a_{i_p};\emptyset;a_{i_{p+1}}) = 1$. So  $V_{\mathbf i,\mathbf j}=1$.
\end{example}
\begin{example}
If $\mathbf j=(i_0,i_{|\mathbf i|+1})$, we have $V_{\mathbf i,\mathbf j}=I_{\mathbf i}$. Hence, any iterated integral is an entry in the variation matrix $V_\mathbf a$ for some $\mathbf a$.
\end{example}

\begin{theorem}\label{thm:DeltaVT}
We have $\Delta V^T=V^T\otimes V^T$.
\end{theorem}
\begin{proof}
Suppose $\mathbf j\leq \mathbf i$. We have
\begin{multline}
\Delta V_{\mathbf i,\mathbf j}=\prod_{p=0}^{|\mathbf j|}\Delta I_{\mathbf i\cap\mathbf j(p)}=\prod_{p=0}^{|\mathbf j|}\Big(\sum_{\mathbf k\leq \mathbf i\cap\mathbf j(p)} I_{\mathbf k}\otimes \prod_{r=0}^{|\mathbf k|} I_{(\mathbf i\cap\mathbf j(p))\cap\mathbf k(r)}\Big)=\\\sum_{\mathbf j\leq\mathbf l\leq\mathbf i}\Big(\prod_{p=0}^{|\mathbf j|} I_{\mathbf l\cap \mathbf j(p)}\otimes \prod_{q=0}^{|\mathbf l|}I_{\mathbf i\cap\mathbf l(q)}\Big)=\sum_{\mathbf j\leq\mathbf l\leq\mathbf i}V_{\mathbf l,\mathbf j}\otimes V_{\mathbf i,\mathbf l}
\end{multline}
This proves the result.
\end{proof}
\begin{corollary}\label{cor:Coassociative}
Goncharov's coproduct is coassociative.\hfill{\qedsymbol}
\end{corollary}
\begin{proof}
It is enough to show that $(\Delta\otimes\id)\Delta V^T=(\id\otimes\Delta)\Delta V^T$. This follows from Theorem~\ref{thm:DeltaVT}, which shows that both sides equal $V^T\otimes V^T\otimes V^T$.
 \end{proof}

\subsection{The coproduct formula for generating series}
In the following we suppose $S$ has a distinguished element $0$. An element $I(0;a_1,\dots,a_d;0)$ is called \emph{degenerate}. Recall Goncharov's generating series \cite{Goncharov_GaloisSymmetriesOfFundamentalGroupoidsAndNoncommutativeGeometry}
\begin{equation}
I(a_0;a_1,\cdots,a_d;a_{d+1}|t_0,\cdots,t_d)=\sum_{n_0,\cdots,n_d\geq0}I(a_0;0^{n_0},a_1,0^{n_1},\cdots,a_d,0^{n_d};a_{d+1})t_0^{n_0}\cdots t_d^{n_d}.
\end{equation}
We shall sometimes write $I_{\mathbf i|\mathbf t}=I(a_0;a_1,\cdots,a_d;a_{d+1}|t_0,\cdots,t_d)$. Consider the map
\begin{equation}
\begin{gathered}
\Gamma\colon\mathcal I(S)\to\mathcal I(S)\\
I_{\mathbf i|\mathbf t}\mapsto \sum_{0\leq p\leq d}I(a_0;a_1,\dots,a_p;0|t_0,\dots,t_p)I(0;a_{p+1},\dots,a_d;a_{d+1}|t_p,\dots,t_d)
\end{gathered}
\end{equation}

\begin{lemma}[{c.f.~\cite[Thm.~5.1]{Goncharov_GaloisSymmetriesOfFundamentalGroupoidsAndNoncommutativeGeometry}}]\label{lemma:GonThm5.1} Modulo degenerates, we have 
\begin{equation}\label{eq:GonThm5.1}
\begin{aligned}
\Gamma\circ\Delta(I_{\mathbf i|\mathbf t})=&\sum_{0=i_0\leq j_0<i_1\leq j_1<\cdots<i_k\leq j_k<i_{k+1}=d+1} I(a_0;a_{i_1},\cdots,a_{i_k};a_{d+1}|t_{j_0},\cdots,t_{j_k})\otimes\prod_{p=0}^k\\
&I(a_{i_p};a_{i_p+1},\cdots,a_{j_p};0|t_{i_p},\cdots,t_{j_p})I(0;a_{j_p+1},\cdots,a_{i_{p+1}-1};a_{i_{p+1}}|t_{j_p},\cdots,t_{i_{p+1}-1})
\end{aligned}
\end{equation}
\end{lemma}
\begin{proof}
A straightforward computation shows that $\Gamma\circ\Delta(I(a_0;0^{n_0},a_1,0^{n_1},\cdots,a_d,0^{n_d};a_{d+1}))$ after eliminating degenerates equals
\begin{multline}
\sum_{0=i_0\leq j_0<\cdots<i_k\leq j_k<i_{k+1}=d+1}I(a_{i_0};0^{m_0},a_{i_1},\cdots,a_{i_k},0^{m_k};a_{i_{k+1}})\otimes\prod_{0\leq p\leq k}\\
\left(\sum_{\substack{r_{j_p}+m_{j_p}+s_{j_p}=n_{j_p}\\r_{j_p},s_{j_p},m_{j_p}\geq0}}I(a_{i_p};0^{n_{i_p}},\cdots,0^{n_{j_p-1}},a_{j_p},0^{r_{j_p}};0)I(0;0^{s_{j_p}},a_{j_p+1},0^{n_{j_p+1}},\cdots,0^{n_{i_{p+1}-1}};a_{i_{p+1}})\right).
\end{multline}
This equals the $t^{n_0}\cdots t^{n_d}$ coefficient of the righthand side of~\eqref{eq:GonThm5.1}. This proves the result.
\end{proof}

\subsection{The Hopf algebra of polylogarithmic iterated integrals}\label{sec:IAndPhi}
Suppose $S$ consists of $0$, $1$, and all elements of the form $\prod_{r=i}^j x_r^{-1}$ with $1\leq i\leq j$. 
Each generator of $\mathcal I(S)$ has the form
\begin{equation}\label{eq:Igenerator} I(a_0;0^{n_0},a_1,0^{n_1},\dots,a_d,0^{n_d};a_{d+1}),\qquad a_1,\dots,a_d\neq 0.
\end{equation}
Goncharov~\cite{Goncharov_MultiplePolylogarithmsAndMixedTateMotives} showed that a realization~\eqref{eq:IteratedIntegral} of an element in $\mathcal I(S)$ can always be expressed in terms of multiple polylogarithms. However, the polylogarithms involved may not be in $\mathbb H^{\symb}$ (e.g.~they may involve $\Li_{\mathbf n}(x_1,x_3)$). We shall thus restrict to a subalgebra. 

\begin{definition}\label{def:PolylogIntegral} An iterated integral~\eqref{eq:Igenerator} is \emph{polylogarithmic} if
for some $i_0<\dots<i_{d+1}$ the equality
\begin{equation}\label{eq:PolylogGeneratorCondition}
    \frac{a_{k+1}}{a_k}=x_{i_k\to i_{k+1}}
\end{equation}
holds for $k=1,\dots,d$, and also for $k=0$ if $a_0\neq 0$. If $a_{d+1}=0$, \eqref{eq:Igenerator} is polylogarithmic if the reverse $I(a_{d+1};0^{n_d},a_d,\dots,a_1,0^{n_0};a_0)$ is polylogarithmic. A degenerate iterated integral~\eqref{eq:Igenerator} is polylogarithmic if replacing $a_0$ or $a_{d+1}$ by 1, would make it polylogarithmic. The polylogarithmic iterated integrals generate a Hopf subalgebra $\mathbb I^{\symb}$ of $\mathcal I(S)$.
\end{definition}

Let $\mathbb I^{\symb}(d)$ denote the Hopf subalgebra generated by the elements that only involve $x_i$ with $i\leq d$. By assigning complex numbers to the $x_i$ the realization \eqref{eq:IteratedIntegral} provides a canonical \emph{realization map}
\begin{equation}\label{eq:ItRealization}
    r\colon\mathbb I^{\symb}(d)\to \Mult(S_d(\C))
\end{equation}
where for a complex manifold $X$, $\Mult(X)$ denotes the set of multivalued holomorphic functions on $X$.
One similarly has a realization map
\begin{equation}\label{eq:LogRelalization}
    r\colon\mathbb \overline{\mathbb H}^{\symb}(d)\to \Mult(S_d(\C)).
\end{equation}

Following Goncharov~\cite{Goncharov_MultiplePolylogarithmsAndMixedTateMotives} we now construct a morphism of Hopf algebras\begin{equation}
\Phi\colon \mathbb I^{\symb}\to\overline{\mathbb H}^{\symb},
\end{equation}
which preserves the realization maps~\eqref{eq:ItRealization} and~\eqref{eq:LogRelalization}.

Writing $I_{\mathbf i|\mathbf t}$ as shorthand for $I(a_0;a_1,\dots,a_d;a_{d+1}|t_0,\dots,t_d)$ we first define
\begin{equation}\label{eq:Phi00}
\Phi(I_{\mathbf i|\mathbf t})=0\text{ if }a_0=a_{d+1}=0,\quad d>0
\end{equation}
\begin{equation}\label{eq:Phi0x}
\Phi(I_{\mathbf i|\mathbf t})=(-1)^d\exp([a_{d+1}]_0t_0)\left[\frac{a_2}{a_1},\dots,\frac{a_{d+1}}{a_d} \middle\vert t_1-t_0,\dots,t_d-t_0\right] \text{ if }a_0=0, a_{d+1}\neq 0,
\end{equation}
\begin{equation}\label{eq:Phix0}
\Phi(I_{\mathbf i|\mathbf t})=(-1)^d\Phi\big(I(0;a_d,\dots,a_1;a_0|-t_d,\dots,-t_1)\big) \text{ if }a_{d+1}=0.
\end{equation}
The general formula for $\Phi$ is then given by
\begin{equation}\label{eq:PhiGeneral}
\Phi(I_{\mathbf i|\mathbf t})=\sum_{0\leq p\leq d}\Phi\big(I(a_0;a_1,\dots,a_p;0|t_0,\dots,t_p)\big)\Phi\big(I(0;a_{p+1},\dots,a_d;a_{d+1}|t_p,\dots,t_d)\big)
\end{equation}
with the righthand side defined by the formulas~\eqref{eq:Phi00}, \eqref{eq:Phi0x} and~\eqref{eq:Phix0}.

\begin{example}\label{ex:PhiOfPolylog} One has
\begin{equation}
\begin{aligned}
\Phi\big(I(0;0^n;a_1)\big)&=\frac{1}{n!}[a_1]_0^n\\ \Phi\big(I(0;\frac{1}{x_1\cdots x_d},0^{n_1-1},\frac{1}{x_2\cdots x_d},\dots,\frac{1}{x_d},0^{n_d-1};1)\big)&=(-1)^d[x_1,\dots,x_d]_{n_1,\dots,n_d}\\
\Phi\big(I(0;0^{n_0-1},a_1,0^{n_1-1},\dots,a_d,0^{n_d-1};a_{d+1})\big)&=\\
(-1)^{n_0+d-1}\sum_{i_0+\dots+i_d=n_0-1}(-1)^{i_0}\frac{[a_{d+1}]_0^{i_0}}{i_0!}\binom{n_1+i_1-1}{n_1-1}&\dots\binom{n_d+i_d-1}{n_d-1}\left[\frac{a_2}{a_1},\dots,\frac{a_{d+1}}{a_d}\right]_{n_1+i_1,\dots,n_d+i_d}.
\end{aligned}
\end{equation}
\end{example}


\begin{proposition}[{\cite[Prop.~2.15]{Goncharov_MultiplePolylogarithmsAndMixedTateMotives}}]
The map $\Phi$ respects realization, i.e.~we have a commutative diagram
\begin{equation}
    \cxymatrix{{\mathbb I^{\symb}(d)\ar[rr]^-{\Phi}\ar[rd]^-{r}&&\overline{\mathbb H}^{\symb}(d)\ar[ld]_-{r}\\&\Mult(S_d(\C))&}}
\end{equation}
\end{proposition}

\begin{proposition}\label{prop:PhiPreservesCoproduct}
$\Phi$ is a Hopf algebra morphism, i.e.~it preserves the coproduct.
\end{proposition}
\begin{proof} It follows from~\eqref{eq:PhiGeneral} that $\Phi\circ\Gamma=\Phi$, so we must prove that $\Phi\circ\Gamma\circ\Delta=\Delta\circ\Phi$.
It is enough to verify this in the cases when $a_0=0$ and $a_{d+1}=0$, respectively. We prove the case $a_0=0$ and leave the other to the reader. Using the formula for $\Gamma\circ\Delta$ in Lemma~\ref{lemma:GonThm5.1} we see that $\Phi\circ\Gamma\circ\Delta(I_{\mathbf i|\mathbf t})$ equals
\begin{equation}\label{eq:PhiDelta}
\begin{aligned}
&\sum_{1\leq i_1\leq j_1<\cdots<i_k\leq j_k<i_{k+1}=d+1} \Phi(I(0;a_{i_1},\cdots,a_{i_k};a_{d+1}|t_{0},t_{j_1},\cdots,t_{j_k}))\otimes\Phi(I(0;a_{1},\cdots;a_{i_{1}}|t_{0},\cdots,t_{i_{1}-1}))\\
&\hspace*{3pc}\prod_{p=1}^k\Phi(I(a_{i_p};a_{i_p+1},\cdots,a_{j_p};0|t_{i_p},\cdots,t_{j_p}))\Phi(I(0;a_{j_p+1},\cdots;a_{i_{p+1}}|t_{j_p},\cdots,t_{i_{p+1}-1}))\\
&=\sum_{1\leq i_1\leq j_1<\cdots<i_k\leq j_k<i_{k+1}=d+1}(-1)^k\exp([a_{d+1}]_0t_0)\left[\frac{a_{i_2}}{a_{i_1}},\cdots,\frac{a_{i_{k+1}}}{a_{i_k}}|t_{j_1}-t_0,\cdots,t_{j_k}-t_0\right]\otimes\\
&\hspace*{3pc}(-1)^{i_1-1}\exp([a_{i_1}]_0t_0)\left[\frac{a_2}{a_1},\cdots,\frac{a_{i_1}}{a_{i_1-1}}|t_1-t_0,\cdots,t_{i_{1}-1}-t_0\right]\prod_{p=1}^k\\
&\hspace*{6pc}\exp(-[a_{i_p}]_0{t_{j_p}})\left[\frac{a_{j_p-1}}{a_{j_p}},\dots,\frac{a_{i_p}}{a_{i_p+1}}|t_{j_p}-t_{j_p-1},\cdots,t_{j_p}-t_{i_p}\right]\\
&\hspace*{6pc}(-1)^{i_{p+1}-j_p-1}\exp([a_{i_p+1}]_0t_{j_p})\left[\frac{a_{j_p+2}}{a_{j_p+1}},\dots,\frac{a_{i_p+1}}{a_{i_p}}|t_{j_p+1}-t_{j_p},\cdots,t_{i_{p+1}-1}-t_{j_p}\right].
\end{aligned}
\end{equation}
Using that $\Delta(\exp([x]_0t))=\exp([x]_0t)\otimes\exp([x]_0t)$ it follows that $\Delta\circ\Phi(I_{\mathbf i|\mathbf t})$ equals
\begin{equation}\label{eq:DeltaPhi}
\begin{aligned}
&(-1)^d\Delta\left(\exp([a_{d+1}]_0t_0)\left[\frac{a_2}{a_1},\cdots,\frac{a_{d+1}}{a_d}|t_1-t_0,\cdots,t_d-t_0\right]\right)\\
&=\sum_{1\leq i_1\leq j_1<\cdots<i_k\leq j_k<i_{k+1}=d+1}(-1)^d\exp([a_{d+1}]_0t_0)\left[\frac{a_{i_2}}{a_{i_1}},\cdots,\frac{a_{i_{k+1}}}{a_{i_k}}|t_{j_1}-t_0,\cdots,t_{j_k}-t_0\right]\otimes\\
&\hspace*{3pc}\exp([a_{d+1}]_0t_0)\left[\frac{a_2}{a_1},\cdots,\frac{a_{i_1}}{a_{i_1-1}}|t_1-t_0,\cdots,t_{i_1-1}-t_0\right]\prod_{p=1}^k\\
&\hspace*{6pc}(-1)^{j_p-i_p}\exp\left(\left[\frac{a_{i_{p+1}}}{a_{i_p}}\right]_0(t_{j_p}-t_0)\right)\left[\frac{a_{j_p-1}}{a_{j_p}},\dots,\frac{a_{i_p}}{a_{i_p+1}}|t_{j_p}-t_{j_p-1},\cdots,t_{j_p}-t_{i_p}\right]\\
&\hspace*{6pc}\left[\frac{a_{j_p+2}}{a_{j_p+1}},\dots,\frac{a_{i_p+1}}{a_{i_p}}|t_{j_p+1}-t_{j_p},\cdots,t_{i_{p+1}-1}-t_{j_p}\right].
\end{aligned}
\end{equation}
Comparing \eqref{eq:DeltaPhi} and \eqref{eq:PhiDelta} we conclude that $\Phi\circ\Gamma\circ\Delta(I_{\mathbf i|\mathbf t})=\Delta\circ\Phi(I_{\mathbf i|\mathbf t})$ as desired.
\end{proof}


\begin{corollary}
The coproduct on $\overline{\mathbb H}^{\symb}$ is coassociative. 
\end{corollary}

Coassociativity of the coproduct on $\mathbb H^{\symb}$ then follows from the result below. Its proof is technical, so we relegate it to Section~\ref{sec:DeltaCommutesWithINV}.
\begin{proposition}\label{prop:INV}
The map $\INV\colon\overline{\mathbb H}^{\symb}\to\mathbb H^{\symb}$ is a homomorphism of Hopf algebras, i.e.~$\Delta_{\mathbb H}\circ\INV=\INV\circ\Delta_{\overline{\mathbb H}}$.
\end{proposition}

\section{The proof that coproduct commutes with inversion}\label{sec:DeltaCommutesWithINV}
We now prove Proposition~\ref{prop:INV} concluding the proof of Theorem~\ref{thm:HHopf}. Clearly, $\Delta\circ\INV=\INV\circ\Delta$ holds for the regular terms, so we consider only inverse terms. Assume this holds for lower depth (the depth one case $\INV\circ\Delta[y^{-1}|-t]=\Delta\circ\INV[y^{-1}|-t]$ is elementary). By the definition of $\INV$ (\ref{eq:GoncharovInversionFormula}) one has  
$\INV [y_d^{-1},\dots,y_1^{-1}|-t_d,\dots,-t_1] = \sum \INV A_i$ with all $A_i$ of lower depth. By induction, we thus have
\begin{equation}
 \Delta \circ \INV [y_d^{-1}\cdots,y_1^{-1}|-t_d,\cdots,-t_1] = \sum \Delta \circ \INV A_i = \sum \INV \circ \Delta A_i,
\end{equation}
so it suffices to show that $\sum \INV \circ \Delta A_i = \INV \circ \Delta [y_d^{-1},\dots,y_1^{-1}|-t_d,\dots,-t_1] $. Rearranging the terms this is equivalent to:
\begin{equation}\label{eq:Claim}
\begin{aligned}
0=\INV\circ\Delta\left(\sum_{j=0}^d\right.&(-1)^j[y_j^{-1},\cdots,y_1^{-1}|-t_j,\cdots,-t_1][y_{j+1},\cdots,y_d|t_{j+1},\cdots,t_d]\\
&+\sum_{j=1}^d\frac{(-1)^j}{t_j}[y_{j-1}^{-1},\cdots,y_1^{-1}|-t_{j-1},\cdots,-t_1][y_{j+1},\cdots,y_d|t_{j+1},\cdots,t_d]\\
&-\sum_{j=1}^d\frac{(-1)^{j}}{t_j}[y_{j-1}^{-1},\cdots,y_1^{-1}|t_j-t_{j-1},\cdots,t_j-t_1]\\
&\phantom{sdfsdfsfsdfsf}\exp([y_{1\to d+1}]_0t_j)[y_{j+1},\cdots,y_d|t_{j+1}-t_{j},\cdots,t_d-t_j]\left.\vphantom{\sum_{j=0}^d(-1)}\right).
\end{aligned}
\end{equation}

We must prove~\eqref{eq:Claim}. We write the righthand side of~\eqref{eq:Claim} as $\INV\circ\Delta(A+B-C)$ and rewrite $\exp([X]_0t)$ as $X^t$. 
We first rewrite $\Delta(B)$ as
\begin{align*}
\sum_{r=1}^d&\frac{(-1)^r}{t_r}\Delta[y_{r-1,\dots,1}^{-1}|-t_{r-1,\dots,1}]\Delta[y_{r+1,\dots,d}|t_{r+1,\dots,d}]\\
=&\sum_{r=1}^d\frac{(-1)^r}{t_r}\sum_{1=i_0\leq j_0<\dots<i_q\leq r<i_{q+1}<\dots<i_{k+1}=d+1}[y_{i_{q-1}\to i_q,\dots,i_0\to i_1}^{-1}|-t_{j_{q-1},\dots,j_0}]\\
&\hphantom{\sum_{r=1}}[y_{i_{q+1}\to i_{q+2},\dots,i_k\to i_{k+1}}|t_{j_{q+1},\dots,k}]\otimes[y_{r-1,\dots,i_q}^{-1}|-t_{r-1,\dots,i_q}][y_{r+1,\dots,i_{q+1}-1}|t_{r+1,\dots,i_{q+1}-1}]\\
&\hphantom{\sum_{r=1}}\prod_{p=0}^{q-1}(-1)^{j_p-i_{p+1}+1}y_{i_p\to i_{p+1}}^{t_{j_p}}[y_{j_p+1,\dots,i_{p+1}-1}|t_{j_p+1,\dots,i_{p+1}-1}-t_{j_p}][y^{-1}_{j_p-1,\dots,i_p}|t_{j_p}-t_{j_p-1,\dots,i_p}]\\
&\hphantom{\sum_{r=1}}\prod_{p=q+1}^k(-1)^{j_p-i_p}y_{i_p\to i_{p+1}}^{t_{j_p}}[y_{j_p-1,\dots,i_p}^{-1}|t_{j_p}-t_{j_p-1,\dots,i_p}][y_{j_p+1,\dots,i_{p+1}-1}|t_{j_p+1,\dots,i_{p+1}-1}-t_{j_p}],
\end{align*}

which simplifies to
\begin{equation}
\begin{aligned}
&\sum_{1=i_0\leq j_0<\dots<i_q\leq j_q<i_{q+1}<\dots<i_{k+1}=d+1}\frac{(-1)^{j_q}}{t_{j_q}}[y_{i_{q-1}\to i_q,\dots,i_0\to i_1}^{-1}|-t_{j_{q-1},\dots,j_0}]\\
&[y_{i_{q+1}\to i_{q+2},\dots,i_k\to i_{k+1}}|t_{j_{q+1},\dots,j_k}]\otimes[y_{j_q-1,\dots,i_q}^{-1}|-t_{j_q-1,\dots,i_q}][y_{j_q+1,\dots,i_{q+1}-1}|t_{j_q+1,\dots,i_{q+1}-1}](-1)^{i_q+q+1}\\
&\prod_{0\leq p\leq k,p\neq q}(-1)^{j_p-i_p}y_{i_p\to i_{p+1}}^{t_{j_p}}[y_{j_p-1,\dots,i_p}^{-1}|t_{j_p}-t_{j_p-1,\dots,i_p}][y_{j_p+1,\dots,i_{p+1}-1}|t_{j_p+1,\dots,i_{p+1}-1}-t_{j_p}].
\end{aligned}
\end{equation}
Similarly, $\Delta(C)$ can be simplified to
\begin{equation}
\begin{aligned}
&\sum_{1\leq i_0\leq j_0<\dots<i_q\leq j_q<i_{q+1}<\dots<i_{k+1}=d+1}\frac{(-1)^{j_q}}{t_{j_q}}[y^{-1}_{i_{q-1}\to i_q,\dots,i_0\to i_1}|t_{j_q}-t_{j_{q-1},\dots,j_0}]y_{1\to d+1}^{t_{j_q}}\\
&[y_{i_{q+1}\to i_{q+2},\dots,i_k\to i_{k+1}}|t_{j_{q+1},\dots,j_k}-t_{j_q}]\otimes[y_{j_q-1,\dots,i_q}^{-1},t_{j_q}-t_{j_q-1,\dots,i_q}]y_{1\to d+1}^{t_{j_q}}\\
&[y_{j_q+1,\dots,i_{q+1}-1}|t_{j_q+1,\dots,i_{q+1}-1}-t_{j_q}](-1)^{i_q+q+1}\\
&\prod_{0\leq p\leq k,p\neq q}(-1)^{j_p-i_p}y_{i_p\to i_{p+1}}^{t_{j_p}-t_{j_q}}[y_{j_p-1,\dots,i_p}^{-1}|t_{j_p}-t_{j_p-1,\dots,i_p}][y_{j_p+1,\dots.i_{p+1}-1}|t_{j_p+1,\dots.i_{p+1}-1}-t_{j_p}]
\end{aligned}
\end{equation}
Finally, $\Delta(A)$ equals
\begin{equation}
\begin{aligned}
\sum_{r=0}^d&(-1)^r\Delta[y_{r,\dots,1}^{-1}|-t_{r,\dots,1}]\Delta[y_{r+1,\dots,d}|t_{r+1,\dots,d}]\\
=&\sum_{r=0}^d(-1)^r\sum_{1=i_0\leq j_0<\cdots<i_q\leq r+1\leq i_{q+1}<\cdots<i_{k+1}=d+1}[y_{i_{q-1}\to i_q,\dots,i_0\to i_1}^{-1}|-t_{j_{q-1},\dots,j_0}]\\
&[y_{i_{q+1}\to i_{q+2},\dots,i_k\to i_{k+1}}|t_{j_{q+1},\dots,j_k}]\otimes[y_{r,\dots,i_q}^{-1}|-t_{r,\dots,i_q}][y_{r+1,\dots,i_{q+1}-1}|t_{r+1,\dots,i_{q+1}-1}]\\
&\prod_{p=0}^{q-1}(-1)^{j_p-i_{p+1}+1}y_{i_p\to i_{p+1}}^{t_{j_p}}[y_{j_p+1,\dots,i_{p+1}-1}|t_{j_p+1,\dots,i_{p+1}-1}-t_{j_p}][y_{j_p-1,\dots,i_p}^{-1}|t_{j_p}-t_{j_p-1,\dots,i_p}]\\
&\prod_{p=q+1}^k(-1)^{j_p-i_p}y_{i_p\to i_{p+1}}^{t_{j_p}}[y_{j_p-1,\dots,i_p}^{-1}|t_{j_p}-t_{j_p-1,\dots,i_p}][y_{j_p+1,\dots,i_{p+1}-1}|t_{j_p+1,\dots,i_{p+1}-1}-t_{j_p}],
\end{aligned}
\end{equation}
which simplifies to
\begin{equation}\label{eq: Induction hypothesis usage}
\begin{aligned}
&\sum_{1=i_0\leq j_0<\cdots<i_q\leq i_{q+1}<\cdots<i_{k+1}=d+1}[y_{i_{q-1}\to i_q,\dots,i_0\to i_1}^{-1}|-t_{j_{q-1},\dots,j_0}]\\
&\hspace*{4pc}[y_{i_{q+1}\to i_{q+2},\dots,i_k\to i_{k+1}}|t_{j_{q+1},\dots,j_k}]\bigotimes(-1)^q\\
&\hspace*{4pc}\left(\sum_{i_q\leq r+1\leq i_{q+1}}(-1)^{r-i_q+1}[y_{r,\dots,i_q}^{-1}|-t_{r,\dots,i_q}][y_{r+1,\dots,i_{q+1}-1}|t_{r+1,\dots,i_{q+1}-1}]\right)\\
&\hspace*{4pc}\prod_{0\leq p\leq k,p\neq q}(-1)^{j_p-i_p}y_{i_p\to i_{p+1}}^{t_{j_p}}[y_{j_p-1,\dots,i_p}^{-1}|t_{j_p}-t_{j_p-1,\dots,i_p}][y_{j_p+1,\dots,i_{p+1}-1}|t_{j_p+1,\dots,i_{p+1}-1}-t_{j_p}].
\end{aligned}
\end{equation}
We split the sum into two parts depending on whether or not $i_q=i_{q+1}$:
\begin{equation}
\sum_{1=i_0\leq j_0<\cdots<i_q<i_{q+1}<\cdots<i_{k+1}=d+1},\quad\sum_{1=i_0\leq j_0<\cdots<i_q=i_{q+1}<\cdots<i_{k+1}=d+1}
\end{equation}
We then apply $\INV$ and use induction on the bracket of~\eqref{eq: Induction hypothesis usage}. The first sum becomes
\begin{equation}
\begin{aligned}
&\INV\left\{\sum_{1=i_0\leq j_0<\cdots<i_q< i_{q+1}<\cdots<i_{k+1}=d+1}[y_{i_{q-1}\to i_q,\dots,i_0\to i_1}^{-1}|-t_{j_{q-1},\dots,j_0}]\right.\\
&[y_{i_{q+1}\to i_{q+2},\dots,i_k\to i_{k+1}}|t_{j_{q+1},\dots,j_k}]\bigotimes(-1)^q\\
&\left(-\sum_{i_q\leq r\leq i_{q+1}-1}\frac{(-1)^{r-i_q+1}}{t_r}[y_{r-1,\dots,i_q}^{-1}|-t_{r-1,\dots,i_q}][y_{r+1,\dots,i_{q+1}-1}|t_{r+1,\dots,i_{q+1}-1}]\right.\\
&+\left.\sum_{i_q\leq r\leq i_{q+1}-1}\frac{(-1)^{r-i_q+1}}{t_r}[y_{r-1,\dots,i_q}^{-1}|t_r-t_{r-1,\dots,i_q}]y_{i_q\to i_{q+1}}^{t_r}[y_{r+1,\dots,i_{q+1}-1}|t_{r+1,\dots,i_{q+1}-1}-t_r]
\right)\\
&\left.\prod_{0\leq p\leq k,p\neq q}(-1)^{j_p-i_p}y_{i_p\to i_{p+1}}^{t_{j_p}}[y_{j_p-1,\dots,i_p}^{-1}|t_{j_p}-t_{j_p-1,\dots,i_p}][y_{j_p+1,\dots,i_{p+1}-1}|t_{j_p+1,\dots,i_{p+1}-1}-t_{j_p}]\right\}.
\end{aligned}
\end{equation}
This equals
\begin{equation}
\begin{aligned}
&=-\INV\left\{\sum_{1=i_0\leq j_0<\cdots<i_q\leq j_q<i_{q+1}<\cdots<i_{k+1}=d+1}\frac{(-1)^{j_q-i_q+q+1}}{t_{j_q}}[y_{i_{q-1}\to i_q,\dots,i_0\to i_1}^{-1}|-t_{j_{q-1},\dots,j_0}]\right.\\
&[y_{i_{q+1}\to i_{q+2},\dots,i_k\to i_{k+1}}|t_{j_{q+1},\dots,j_k}]\otimes[y_{j_q-1,\dots,i_q}^{-1}|-t_{j_q-1,\dots,i_q}][y_{j_q+1,\dots,i_{q+1}-1}|t_{j_q+1,\dots,i_{q+1}-1}]\\
&\left.\prod_{0\leq p\leq k,p\neq q}(-1)^{j_p-i_p}y_{i_p\to i_{p+1}}^{t_{j_p}}[y_{j_p-1,\dots,i_p}^{-1}|t_{j_p}-t_{j_p-1,\dots,i_p}][y_{j_p+1,\dots,i_{p+1}-1}|t_{j_p+1,\dots,i_{p+1}-1}-t_{j_p}]\right\}\\
&+\INV\left\{\sum_{1=i_0\leq j_0<\cdots<i_q\leq j_q<i_{q+1}<\cdots<i_{k+1}=d+1}\frac{(-1)^{j_q-i_q+q+1}}{t_{j_q}}[y_{i_{q-1}\to i_q,\dots,i_0\to i_1}^{-1}|-t_{j_{q-1},\dots,j_0}]\right.\\
&[y_{i_{q+1}\to i_{q+2},\dots,i_k\to i_{k+1}}|t_{j_{q+1},\dots,j_k}]\otimes[y_{j_q-1,\dots,i_q}^{-1}|t_{j_q}-t_{j_q-1,\dots,i_q}][y_{j_q+1,\dots,i_{q+1}-1}|t_{j_q+1,\dots,i_{q+1}-1}-t_{j_q}]\\
&\left.y_{i_q\to i_{q+1}}^{t_{j_q}}\prod_{\substack{0\leq p\leq k\\p\neq q}}(-1)^{j_p-i_p}y_{i_p\to i_{p+1}}^{t_{j_p}}[y_{j_p-1,\dots,i_p}^{-1}|t_{j_p}-t_{j_p-1,\dots,i_p}][y_{j_p+1,\dots,i_{p+1}-1}|t_{j_p+1,\dots,i_{p+1}-1}-t_{j_p}]\right\},
\end{aligned}
\end{equation}
which we write as $-\INV(T_1)+\INV(T_2)$. The second sum becomes
\begin{equation}
\begin{aligned}
&\INV\left\{\sum_{1=i_1\leq j_1<\dots<i_{k+1}=d+1}\Bigg(\right.\\
&\sum_{0\leq q\leq k}(-1)^q[y_{i_q\to i_{q+1},\dots,i_1\to i_2}^{-1}|-t_{j_q,\dots,j_1}][y_{i_{q+1}\to i_{q+2},\dots,i_k\to i_{k+1}}|t_{j_{q+1},\dots,j_k}]\Bigg)\\
&\left.\otimes\prod_{1\leq p\leq k}(-1)^{j_p-i_p}y_{i_p\to i_{p+1}}^{t_{j_p}}[y_{j_p-1,\dots,i_p}^{-1}|t_{j_p}-t_{j_p-1,\dots,i_p}][y_{j_p+1,\dots,i_{p+1}-1}|t_{j_p+1,\dots,i_{p+1}-1}-t_{j_p}]\right\}\\
&=\INV\left\{\sum_{1=i_1\leq j_1<\dots<i_{k+1}=d+1}\Bigg(\right.\\
&-\sum_{1\leq q\leq k}\frac{(-1)^q}{t_{j_q}}[y_{i_{q-1}\to i_q,\dots,i_1\to i_2}^{-1}|-t_{j_{q-1},\dots,j_1}][y_{i_{q+1}\to i_{q+2},\dots,i_k\to i_{k+1}}|t_{j_{q+1},\dots,j_k}]\\
&+\sum_{1\leq q\leq k}\frac{(-1)^q}{t_{j_q}}[y_{i_{q-1}\to i_q,\dots,i_1\to i_2}^{-1}|t_{j_q}-t_{j_{q-1},\dots,j_1}]y_{1\to d+1}^{t_{j_q}}[y_{i_{q+1}\to i_{q+2},\dots,i_k\to i_{k+1}}|t_{j_{q+1},\dots,j_k}-t_{j_q}]\Bigg)\\
&\left.\otimes\prod_{1\leq p\leq k}(-1)^{j_p-i_p}y_{i_p\to i_{p+1}}^{t_{j_p}}[y_{j_p-1,\dots,i_p}^{-1}|t_{j_p}-t_{j_p-1,\dots,i_p}][y_{j_p+1,\dots,i_{p+1}-1}|t_{j_p+1,\dots,i_{p+1}-1}-t_{j_p}]\right\}.
\end{aligned}
\end{equation}
This equals
\begin{equation}
\begin{aligned}
&=-\INV\left\{\sum_{1=i_0\leq j_0<\dots<i_{k+1}=d+1}\frac{(-1)^q}{t_{j_q}}[y_{i_{q-1}\to i_q,\dots,i_0\to i_1}^{-1}|-t_{j_{q-1},\dots,j_0}][y_{i_{q+1}\to i_{q+2},\dots,i_k\to i_{k+1}}|t_{j_{q+1},\dots,j_k}]\right.\\
&\left.\otimes\prod_{0\leq p\leq k}(-1)^{j_p-i_p}y_{i_p\to i_{p+1}}^{t_{j_p}}[y_{j_p-1,\dots,i_p}^{-1}|t_{j_p}-t_{j_p-1,\dots,i_p}][y_{j_p+1,\dots,i_{p+1}-1}|t_{j_p+1,\dots,i_{p+1}-1}-t_{j_p}]\right\}\\
&+\INV\left\{\sum_{1=i_0\leq j_0<\dots<i_{k+1}=d+1}\frac{(-1)^q}{t_{j_q}}[y_{i_{q-1}\to i_q,\dots,i_0\to i_1}^{-1}|t_{j_q}-t_{j_{q-1},\dots,j_0}]y_{1\to d+1}^{t_{j_q}}\right.\\
&[y_{i_{q+1}\to i_{q+2},\dots,i_k\to i_{k+1}}|t_{j_{q+1},\dots,j_k}-t_{j_q}]\\
&\left.\otimes\prod_{0\leq p\leq k}(-1)^{j_p-i_p}y_{i_p\to i_{p+1}}^{t_{j_p}}[y_{j_p-1,\dots,i_p}^{-1}|t_{j_p}-t_{j_p-1,\dots,i_p}][y_{j_p+1,\dots,i_{p+1}-1}|t_{j_p+1,\dots,i_{p+1}-1}-t_{j_p}]\right\}
\end{aligned}
\end{equation}
which we write as $-\INV(T_3)+\INV(T_4)$. We note that 
\begin{equation}
    \Delta(B)=T_1,\qquad \Delta(C)=T_4,\qquad T_2=T_3.
\end{equation}
This implies that $\INV\circ\Delta(A+B-C)=0$, proving the claim.

\bibliographystyle{alpha}
\bibliography{Bibliography}

\end{document}